%% file: ConvexSupported.tex
\documentclass[11pt,a4paper]{article}

\usepackage[T1]{fontenc}
\usepackage{lmodern}

\usepackage{enumerate}
\usepackage{graphicx}
\usepackage{algorithm}
\usepackage{algorithmic}
\usepackage{amsmath,amssymb}
\usepackage{amsthm}
\usepackage{multirow}
\usepackage{rotating}
\usepackage{a4wide}
\usepackage{pdflscape} 
\usepackage[round]{natbib}

\DeclareMathOperator{\conv}{\rm conv}

\DeclareMathOperator{\fr}{\rm bd}
\DeclareMathOperator{\cl}{\rm cl}
\newcommand{\mR}{\mathbb{R}}

\newcommand{\mZ}{\mathbb{Z}}

\newtheorem{prop}{Proposition}
\newtheorem{Lem}{Lemma}

\newtheorem{Def}{Definition}
\newtheorem{ex}{Example}

\newcommand{\rien}[1]{}

\usepackage{marginnote}
\usepackage[usenames,dvipsnames]{color}

\usepackage[normalem]{ulem}

\newcommand{\ext}{{{\rm ext}}}
\newcommand{\init}{{{\rm init}}}

\usepackage{framed} 
\title{A Simple and Efficient Dichotomic Search Algorithm for Multi-Objective Mixed Integer Linear Programmes}

\author{
  Anthony Przybylski$^1$, Kathrin Klamroth$^2$, Renaud Lacour$^2$\\
  (1) Universit\'e de Nantes\\
  LS2N UMR CNRS 6004\\
  2 Rue de la Houssini\`ere BP 92208\\
  44322 Nantes Cedex 03 -- France\\
  Anthony.Przybylski@univ-nantes.fr\\
\and
  (2) University of Wuppertal\\
  Optimization and Approximation\\
  Department C - Mathematics und Natural Sciences\\
  Gau{\ss}str.~20, 
  42119 Wuppertal -- Germany\\ 
  \{klamroth,lacour\}@math.uni-wuppertal.de\\
}

\date{\today}

\begin{document}

\maketitle

%
%

\begin{abstract}
We present a simple and at the same time fficient algorithm to compute all nondominated extreme points in the outcome set of multi-objective mixed integer linear programmes in any dimension. 
The method generalizes the well-known dichotomic scheme to compute the set of nondominated extreme points in the outcome set of a bi-objective programme based on the iterative solution of weighted sum scalarizations. It uses as a main routine a convex hull algorithm. 
The algorithm is illustrated with, and numerically tested on, instances of multi-objective assignment and knapsack problems. Experimental results confirm the computational efficiency of the approach. 
Finally, an implementation in incremental polynomial time with respect to the number of computed nondominated extreme points is possible, under the assumption that the lexicographic version of the problem can be solved in polynomial time.

\vspace{1ex} \noindent

\end{abstract}

\input{introduction.tex}

\input{biobjective.tex}
\input{multiobjective.tex}
\input{approximation.tex}
\input{newmethod}
\input{experiments.tex}

\section{Conclusion}

In this paper,  we have presented a new method to compute all the nondominated extreme points of a multi-objective (mixed)-integer linear programme. Our method generalizes the classical dichotomic scheme of \citep{aneja79,cohon78} to the multi-objective case in the most natural way. This has been realized by a combination of ideas that were rather complex to apply individually and that have been simplified, resulting in a method far less technical than the preceding propositions. The practical efficiency of the method is supported by experimental results. On a theoretical point of view, an implementation in incremental polynomial time with respect to the number of nondominated extreme points is possible, if the lexicographic problem can be solved in polynomial time.

\section*{Acknowledgments}
This work was partially supported by the bilateral cooperation project MOCO-BBC funded by the German Academic Exchange Service (DAAD, Project-ID 57212018) and by Campus France, PHC PROCOPE 2016 (Project No.~35476QD). 

This work was partially supported by the project ANR/DFG-14-CE35-0034-01 ``Exact Efficient Solution of Mixed Integer Programming Problems with Multiple Objective Functions (vOpt)''.

\bibliographystyle{plainnat}
\bibliography{biblioSE}

\end{document}

%% file: introduction.tex
\section{Introduction}\label{intro}

Even though efficient solution methods exist for many classes of single-objective integer linear programming problems (like, for example, knapsack and assignment problems), this is in general not the case if multiple objective functions have to be considered. Multi-objective integer linear programming problems are often intractable (and hence possess a possibly huge number of efficient solutions), and the majority of efficient solutions may be unsupported, \citep[see, for example,][]{MEXG2002}. To nevertheless obtain information on the structure of a given problem and on the related trade-offs, the iterative solution of weighted sum scalarizations is an indispensible tool in many algorithms. This includes, for example, two phase methods as well as branch and bound and dynamic programming based algorithms \citep{Ehrgott2016}.
Weighted sum scalarizations principally allow the computation of all nondominated points on the boundary of the convex hull of the set of feasible outcome vectors, and in particular of its extreme points which will be referred to as \emph{nondominated extreme points} in the following. While the identification of the relevant weights is easy in bi-objective problems, this is in general not the case in higher dimensions. Existing methods are often computationally expensive and/or complicated to use as will be illustrated in the following sections.

The goal of this paper is to devise a simple and at the same time efficient algorithm to compute \emph{all} nondominated extreme points of a multi-objective mixed integer linear programming problem in any dimension.
To keep the exposition simple, we will focus on multi-objective integer linear programmes in the following.
However, most of the results immediately transfer to the case of multi-objective mixed integer linear
programmes. We will indicate this in the following whenever appropriate.

A multi-objective integer linear programme is written as

\parbox{0.85\textwidth}{
$$
 \min\left\{\left(z_1(x),\ldots,z_p(x)\right) = Cx:  x \in X\right\}
$$
}\hfill
\parbox{0.1\textwidth}{\hfill (MOIP)}
where $p \geqq 2$ and $C \in \mR^{p \times n}$. $X$ denotes the set of
feasible solutions of the problem and is defined by 
\begin{equation}\label{feasible}
X = \{x \in \mZ^n : Ax = b,\ x\geqq 0\}, 
\end{equation} 
with $A \in \mR^{m \times n}$ and $b \in \mR^m$. 
Unless stated otherwise, we will assume that all data is integer.
The {\it outcome
set} $Y$ is defined by $Y := \{Cx: x \in X\}$. 

We assume that no feasible solution minimizes all objective functions
simultaneously and that the ideal point $y^I\in\mR^p$ with components
$y_k^I:=\min\{z_k(x): x\in X\}$, $k=1,\dots,p$, exists and is strictly positive, i.e. $y^I_k >0$, $k = 1,\ldots,p$. Note that this is not a restrictive assumption. We use the following notation for componentwise orders
in $\mR^p$. Let  $y^1,y^2 \in \mR^p$. We write $y^1 \leqq y^2$
if $y^1_k \leqq y^2_k$ for $k = 1,\ldots,p$, $y^1 \leq y^2$ if $y^1
\leqq y^2$ and $y^1 \neq y^2$, and $y^1 < y^2$ if $y^1_k < y^2_k, k = 1,\ldots,p$. We define $\mR^p_\geqq := \{x \in \mR^p : x \geqq 0\}$
and analogously $\mR^p_\ge$ and $\mR^p_>$. 

\begin{Def}
A feasible solution $x^* \in X$ is  {\em efficient (weakly efficient)} if 
there does not exist any other feasible solution $x \in X$ such that 
$z(x)  \leq z(x^*)$ ($z(x) < z(x^*)$). If $x^*$ is efficient, then $z(x^*)$ is a {\em nondominated outcome vector (weakly nondominated outcome vector)} or {\em nondominated point (weakly nondominated point)} for short. If $x, x' \in X$ are such that $z(x) \le z(x')$
we say that $x$ dominates $x'$ and $z(x)$ dominates $z(x')$.  Feasible
solutions $x, x' \in X$ are {\em equivalent} if  $z(x) = z(x')$.

The set $X_E$ of all efficient solutions ($X_{wE}$ of all weakly efficient solutions)
and the set $Y_N$ of all nondominated outcome vectors ($Y_{wN}$ of all weakly nondominated outcome vectors)
are referred to as the \emph{efficient set (weakly efficient set)} and the \emph{nondominated set (weakly nondominated set)},
respectively.
\end{Def}

The nondominated set is bounded by two particular vectors:
the ideal point $y^I$ with $y^I_k = \min_{y\in Y_N} y_k$ and
the nadir point $y^N$ with $y^N_k = \max_{y\in Y_N} y_k$,
for $k\in \{1, \dots, p\}$.

Several classes of  efficient solutions can be distinguished.

\begin{itemize}
\item {\em Supported} efficient solutions are optimal solutions of a
  weighted sum single objective problem

  \parbox{0.8\textwidth}{
  $$\min \{ \lambda_1 z_1(x) + \ldots + \lambda_p z_p(x): x\in X\}$$
  }\hfill \parbox{0.1\textwidth}{\hfill (MOIP$_\lambda$)} 

  for some weight $\lambda \in \mathbb{R}^p_>$. Their images in the objective
  space are supported nondominated points. We use the notations $X_{SE}$
  and $Y_{SN}$, respectively. All supported nondominated points are
  located on the boundary of the   convex hull of $Y$ ($\conv Y$),
  i.e., they are nondominated  points of $(\conv Y) + \mathbb{R}_\geqq^p.$ 
  
  \item {\em Nonsupported} efficient solutions are efficient solutions
  that are not optimal solutions of (MOIP$_{\lambda}$) for any $\lambda
  \in \mathbb{R}^p_>$. Nonsupported nondominated points  
  are located in the interior of the convex hull of $Y$. 
\end{itemize}

In addition we can distinguish two classes of supported efficient
solutions, namely
\begin{itemize}
\item supported efficient solutions $x$ whose objective vectors $z(x)$
  are located on the vertex set of $\conv Y$ (we call these \emph{extremal
  supported} efficient solutions, $X_{SE1}$, and \emph{nondominated extreme
  points}, $Y_{SN1}$, respectively) and  
\item those supported efficient solutions $x \in X_{SE}$ for which $z(x)$ is located in the relative
  interior of a face  of $\conv Y$. For such a solution $x$ there
  exist $2\leq \ell\leq p$ extremal supported efficient solutions 
  $x^1, \dots, x^{\ell}$ 
  satisfying $z(x^i) \neq z(x^j)$ for all $i, j \in \{1, \dots, \ell\}$, $i \neq j$,
  and $\alpha \in \mR^{\ell}_>$ with $\sum_{k=1}^{\ell} \alpha_k = 1$
  such that $z(x) = \sum_{k=1}^{\ell} \alpha_k z(x^k)$.
	The corresponding sets of \emph{non-extremal supported} efficient solutions and 
  their outcome vectors are  denoted by  $X_{SE2}$ and $Y_{SN2}$,
  respectively.  
\end{itemize}

For a given subset $S$ of the objective space $\mR^p$, we denote by $S_N$ 
the set of all nondominated points relatively to $S$, 
i.e. $S_N = \{y \in S\, :\, \nexists y'\in S, y' \leq y\}$, and by $S_{wN}$ the set of all weakly nondominated points relatively to $S$, i.e. $S_{wN} = \{y \in S\, :\, \nexists y'\in S, y' < y\}$. 
Let $F$ be a face of a convex polytope $P\subseteq \mR^p$ of dimension $p$. 
We say that $F$ is \emph{nondominated} if $F \subseteq P_N$, and that $F$ is \emph{weakly nondominated} if $F \subseteq P_{wN}$.
If $F$ is a facet, 
i.e.\ a $(p-1)$-dimensional polytope,
then $F$ is nondominated (weakly nondominated) if, and only if, every  
normal vector $\lambda$
to $F$ pointing to the interior of $P$ is such that $\lambda_k > 0$ ($\lambda_k \geq 0$) in minimization problems, and $\lambda_k < 0$ ($\lambda_k \leq 0$) in maximization problems, for all $k \in  \{1, \dots, p\}$. A \emph{maximal} nondominated face is a nondominated face that is not contained in any other nondominated face.

The purpose of this paper is to develop a simple and at the same time efficient dichotomic search algorithm that determines
one efficient solution for each nondominated extreme point,
or equivalently, to generate all the nondominated extreme points in the
outcome set of a multi-objective integer programme. 

When $p=2$ this
reduces to solving a sequence of single objective problems
(MOIP$_\lambda$) with $\lambda\in\mR^2_>$, because the ``natural'' order of nondominated points (i.e. $y_1^r < y_1^s$ implies $y_2^s < y_2^r$) allows
us to search by dichotomy. This dichotomic search provides the
appropriate values of  $\lambda$ in a straightforward way, see
e.g.\ \citet{cohon78}, who calls the procedure ``noninferior set
estimation method'', or \citet{aneja79}. An extension of this procedure to dimensions $p>2$ was stated as a major challenge in \citet{MEXG2002}, and has remained unsolved until recently with the methods proposed by \citet{3SE}, \citet{MSOzpKok} and \citet{boekler15},
which can be applied to compute the nondominated extreme points of any MOIP. 
The proof of their correctness explicitly or implicitly relies on a decomposition of the weight set, i.e., the set of all relevant
weights for the considered weighted sum scalarizations (MOIP$_\lambda$). Using a dual interpretation, \citet{boekler15} show in addition that, if the respective weighted sum scalarizations can be solved in polynomial time, their method runs in output polynomial time (with respect to the number of nondominated extreme points computed) for every fixed number $p$ of objectives. 

The iterative solution of weighted sum scalarizations and
the corresponding decomposition of the weight set that comprises all relevant weights 
play an important role also in the context of multi-objective linear programmes (MOLP), where $x \in \mZ^n$ is replaced by $x \in \mR^n$ in (\ref{feasible}). 
It is well known that every efficient solution of (MOLP) is supported,
\citep[see][]{iser74,Steuer}. 
Algorithms for the solution of MOLPs are available since the seventies with the first kind of methods being multicriteria simplex algorithms \citep[see, for example,][and references therein]{EhrgottWiecekChapter}. 
More recently, another stream of research has been proposed: enumerating points in the objective space $\mR^p$ rather than solutions in the decision space $\mR^n$. \citet{benson97,benson98a} has argued that this is advantageous because in general $p$ is much smaller than $n$. 
Weight set decomposition has been used to prove the correctness of both kinds of methods \citep[see][for the multi-objective simplex algorithm]{yuzeleny}. Algorithms computing a weight set decomposition by determining simultaneously the set of nondominated extreme points of an MOLP have also been proposed. For example, \citet{benson2002} use a decomposition of the weight set described in \citet{benson2000}. 
Weight set decomposition has also been the foundation for the extension of the primal-dual simplex algorithm to the multi-objective case, see \citet{Ehrgott:2007}. 
A dual variant of Benson's outer approximation algorithm  is suggested in \citet{ehrg:adua:2012}. It is based on the formulation of dual problems that include information on weights describing the facets of the primal polyhedron \citep[see][]{heyd:geom:2008,hame:bens:2014}, i.e., on the weight set decomposition. Since weight cells correspond to nondominated extreme points in the objective space, primal dual methods for MOLPs imply a double description of the nondominated set. This interrelation has beeen used in \cite{csirmaz2018inner} to further refine primal dual methods for MOLPs by combining a combinatorial enumeration strategy for the nondominated extreme points (yielding a provably best possible iteration count) with tailored single-objective LP-solver calls.  Note that the approach of \citet{heyd:geom:2008} and \citet{hame:bens:2014} can also be viewed as the basis of the work of \citet{boekler15} for multi-objective combinatorial optimization problems mentioned above. 

In a different line of research, 
there exists a large variety of methods whose purpose is to \emph{approximate} the set of nondominated points of convex or non-convex multi-objective problems \citep[see][for a review]{RW05}. In particular, methods for the approximation of multi-objective convex problems can be used for the approximation of the set of nondominated extreme points of 
MOIPs, see, for example, \citet{SKW02} and \citet{RenDamHer11}.
For many of these methods, an exact representation containing all nondominated extreme points can be obtained by driving the quality indicator to zero.
Even if it is not the original purpose of such approximation methods, this shows that 
many of them can be applied to compute the set of nondominated extreme points of MOIPs. 

In the following, we will combine ideas from approximation methods with the approaches of \citet{3SE}  and \citet{MSOzpKok} to obtain a simple and efficient dichotomic scheme for the exact computation of all nondominated extreme points of MOIPs. Note that we could also make the reverse step, i.e., by combining our new algorithm with appropriate quality indicators, it can be easily converted into a method to approximate the convex hull of the nondominated set of an MOIP.

Since dichotomic search algorithms rely on the repeated solution of weighted sum scalarizations (MOIP$_\lambda$), it is reasonable to suppose for a practical application that the single-objective problems (MOIP$_\lambda$) can be solved efficiently. Therefore, the method we propose would most likely be applied to multi-objective combinatorial optimization problems. 
As in \citet{3SE},  we will use the assignment
 and knapsack problems with $p$ objectives, respectively ($p$AP) and ($p$KP), to illustrate and test our algorithms. 
The multi-objective assignment problem can be stated as

\parbox{0.8\textwidth}{
$$
\begin{array}{rcrcll}
 \min z_k(x) &  = &  \displaystyle {\sum_{i=1}^n \sum_{j=1}^n c_{ij}^k
    x_{ij}}  &    &  & k=1,\ldots,p  \\ 
 & & \displaystyle { \sum_{i=1}^n x_{ij}} & = & 1 & j=1, \ldots, n \\ 
 & & \displaystyle { \sum_{j=1}^n x_{ij}} & = & 1 & i=1, \ldots, n \\ 
 & & x_{ij} & \in & \{0,1\} & i,j = 1, \ldots, n 
 \end{array}
$$
}
\hfill
\parbox{0.1\textwidth}
{\hfill ($p$AP)}
where all objective function coefficients $c_{ij}^k$ are non-negative
integers and $x = (x_{11},\dots,x_{nn})^T\in\{0,1\}^{n^2}$ is the vector of decision 
variables. The multi-objective knapsack problem is given by

\parbox{0.8\textwidth}{
$$
\begin{array}{cllr}
\max z_k(x) & = & \displaystyle{\sum_{i = 1}^n c_i^k x_i} & k = 1,\ldots,p\\
& & \displaystyle{\sum_{i = 1}^n w_ix_i \leq \omega} & \\
& & x_i \in \{0,1\} & i = 1,\ldots,n\\
\end{array}
$$
}
\hfill
\parbox{0.15\textwidth}
{\hfill ($p$KP)}
where all objective function coefficients $c_i^k$ are non-negative
integers, $w_i$ and $\omega$ are positive integers, and $x = (x_1,\dots,x_n)^T\in\{0,1\}^n$ is the vector of decision
variables.  

The remainder of the paper is structured as follows. 
We review the bi-objective dichotomic scheme and the difficulties of its extension to the multi-objective case in Section~\ref{sec:lit2obj}. The methods proposed for the determination of nondominated extreme points of multi-objective (mixed-)integer linear programmes are reviewed in Section~\ref{sec:litpobj}. Solution methods for computing a convex approximation of multi-objective optimization problems are reviewed in Section~\ref{sec:litapprox}. New developments extending ideas proposed in the literature are proposed in Section~\ref{newMethod}, and these developments are immediately used to define two new solution methods. Finally, experimental results are provided on multi-objective assignment and knapsack problems in Section~\ref{sec:exp} and show the practical efficiency of the proposed methods.

%% file: biobjective.tex
\section{Bi-objective dichotomic scheme and difficulties in its extension to the multi-objective case}
\label{sec:lit2obj}  

The bi-objective dichotomic scheme has been designed using specific properties of the bi-objective case. As a consequence, its extension to the multi-objective case is not obvious.

\subsection{Classical dichotomic scheme in the bi-objective case}

The dichotomic scheme is based on the consideration of consecutive supported nondominated points $y^r$ and $y^s$ with respect to one objective, i.e. $y^r_1
< y^s_1$ and $y^r_2 > y^s_2$. A weighted sum problem (MOIP$_\lambda$)
with $\lambda_{1} = y^r_2 - y^s_2 > 0$ and $\lambda_{2} = y^s_1 - y^r_1 > 0$
is solved to find new supported points located ``between'' $y^r$ and $y^s$.
The vector $\lambda$ corresponds to a normal to the line joining $y^r$ and $y^s$, as illustrated in Figure \ref{normal1} with a negative multiple of $\lambda$ to highlight the optimization sense. Following the solution of this weighted sum problem, a supported nondominated point $y^t$ is obtained and two cases are possible.
\begin{enumerate}[(i)]
\item If $\lambda^Ty^t < \lambda^Ty^r$, then $y^t$ is necessarily a new supported nondominated point and two new problems (MOIP$_\lambda$) have to be solved, one with $\lambda$ defined by $y^r$ and $y^t$ and one with $\lambda$
  defined by $y^t$ and $y^s$ (see Figure \ref{normal2}).   
\item If $\lambda^Ty^t = \lambda^Ty^r$, then the search stops, and $[y^r,y^s]$ is a part of an edge of $(\conv Y_{SN})_N$. 
\end{enumerate}
This scheme is initialized with nondominated points minimizing respectively the first and the second objectives, 
and is usually implemented recursively. More detailed descriptions can
be found in \citet{aneja79} and \citet{cohon78}.  

\begin{figure}[htb] 
\begin{center}
  \begin{minipage}[t]{0.45\textwidth}
  \includegraphics[width=45mm]{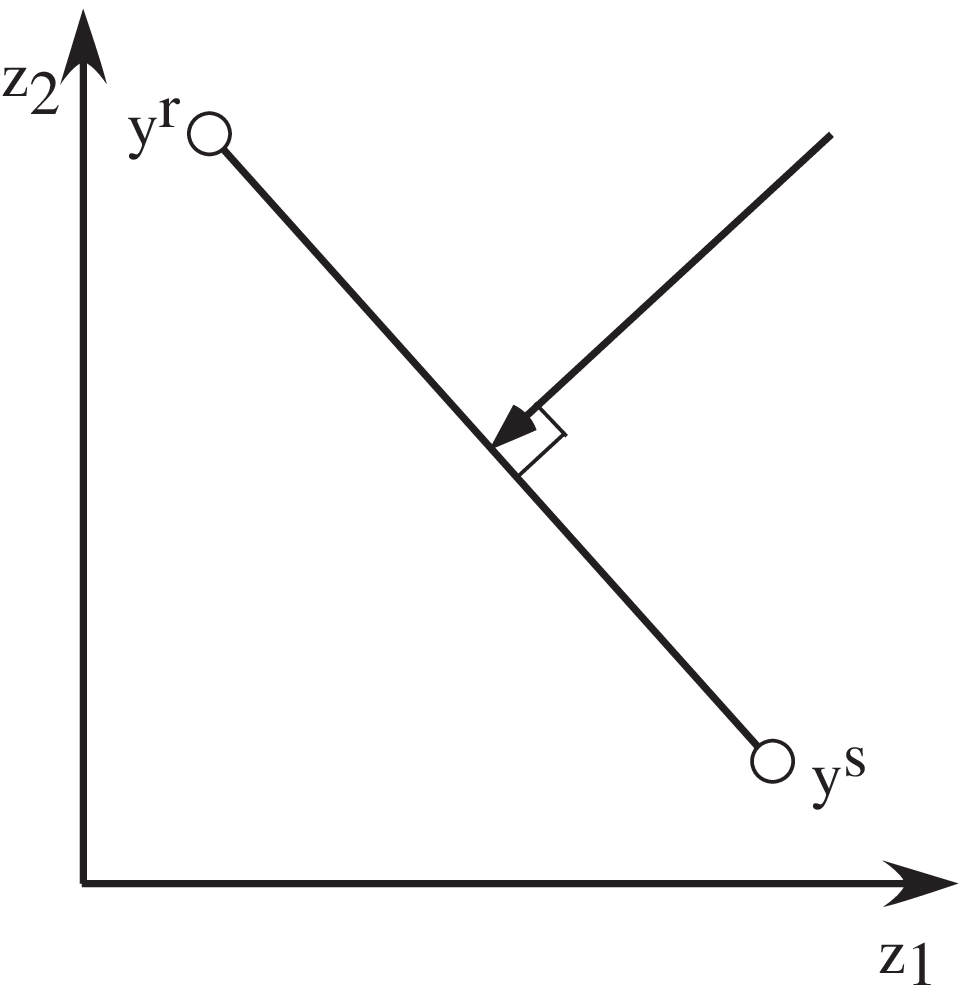}
  \caption{Problem (MOIP$_\lambda$) defined by $y^r$ and $y^s$.}
  \label{normal1} 
  \end{minipage}
  \hfill
  \begin{minipage}[t]{0.45\textwidth}
  \includegraphics[width=45mm]{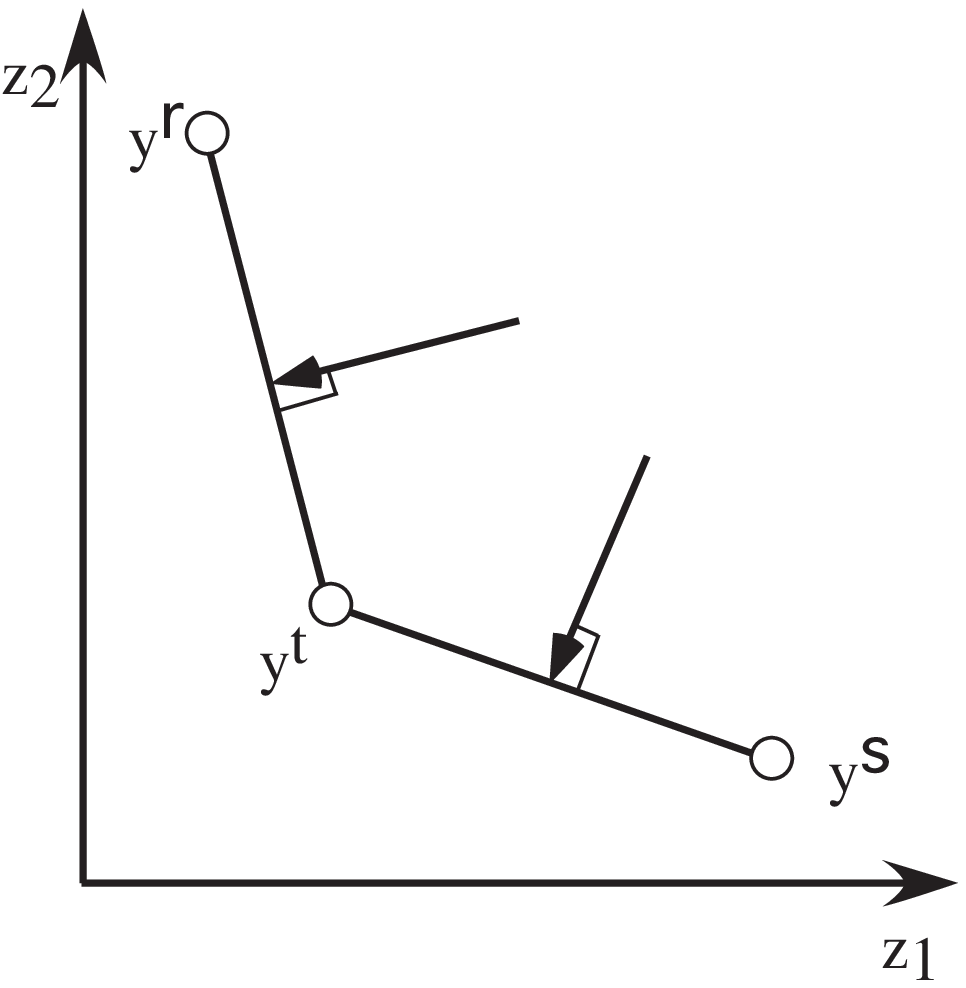}
  \caption{$\lambda^Ty^t < \lambda^Ty^r$, therefore two new
    problems (MOIP$_\lambda$) have to be solved.} 
  \label{normal2}
  \end{minipage}
\end{center}
\end{figure}

\subsection{Difficulty in the extension to the multi-objective case} \label{subsec:difficulties}

The dichotomic scheme is not immediate to generalize to more than two
objectives, because as illustrated in Figures \ref{normal1} and
\ref{normal2}, it relies on the natural order of nondominated points
in the objective space, i.e.\ that $y_1^r < y_1^s$ implies $y_2^s < y_2^r.$
The absence of this natural order causes the following difficulties when $p > 2$ \citep{3SE}. 
\begin{itemize}
\item To define a hyperplane in $\mR^p$,  $p$ points are
  necessary. But there might be more than $p$ different nondominated
  extreme points (at most $p!$ lexicographically optimal points)
  resulting from the initial single objective optimizations. It is
  unclear which points to choose to define a hyperplane to start the
  procedure. 
\item Even if the initialization yields exactly $p$ initial points, the
  normal to a hyperplane defined by $p$ nondominated points does not
  necessarily have positive components. Then the optimization  of
  (MOIP$_\lambda$) does not necessarily yield other supported
  nondominated points.  
\end{itemize}
The following example has been used in \citet{3SE} to illustrate how the dichotomic scheme can fail.

\begin{ex}\label{diff} Consider an instance of the assignment problem with
  three objectives \citep{2phrec} where 
$$
C^1 = \begin{pmatrix}
3 & 6 & 4 & 5\\
2 & 3 & 5 & 4\\
3 & 5 & 4 & 2\\
4 & 5 & 3 & 6\\
\end{pmatrix}
\hbox{, }
C^2 = \begin{pmatrix}
2 & 3 & 5 & 4\\
5 & 3 & 4 & 3\\
5 & 2 & 6 & 4\\
4 & 5 & 2 & 5\\
\end{pmatrix}
\hbox{, and }
C^3 = \begin{pmatrix}
4 & 2 & 4 & 2\\
4 & 2 & 4 & 6\\
4 & 2 & 6 & 3\\
2 & 4 & 5 & 3\\
\end{pmatrix}.
$$ 

Let us try to apply the usual dichotomic scheme to this instance. 
The three single objective assignment problems with objective
coefficients $C^k$, $k=1,2,3$, yield three points:
\begin{description}
\item[$k=1$:] $x^1$ with $x_{11} = x_{22} = x_{34} = x_{43} = 1$ and
  objective vector $y^1 = (11,11,14)^T$ is the unique optimal solution 
  minimizing the first objective. 
\item[$k=2$:] $x^2$ with $x_{11} = x_{24} = x_{32} = x_{43} = 1$ and
  objective vector $y^2 = (15,9,17)^T$ is the unique optimal solution
  minimizing the second objective. 
\item[$k=3$:] $x^3$ with $x_{14} = x_{23} = x_{32} = x_{41} = 1$ and
  objective vector $y^3 = (19,14,10)^T$ is the unique optimal solution
  minimizing the third objective. 
\end{description}

The normal to the plane defined by $y^1$, $y^2$ and $y^3$
is either $\lambda = (1,-40,-28)^T$ or $\lambda = (-1,40,28)^T$. In
both cases not all components are  positive. Solving
{\rm (MOIP$_{(1,-40,-28)}$)} we get the dominated point $y^d = (16,20,16)^T$ 
for $x^d$ with $x_{14} = x_{21} = x_{33} = x_{42} = 1$ as optimal
solution.  Solving {\rm (MOIP$_{(-1,40,28)}$)} we obtain $y^1$, $y^2$ or $y^3$. 
Therefore, the dichotomic scheme stops without finding any further
supported nondominated points. However, $y^4 = (13,16,11)^T$ for $x^4$ with
$x_{13} = x_{22} = x_{34} = x_{41} = 1$ is a supported nondominated
point that can be found solving {\rm (MOIP$_\lambda$)} with $\lambda = (1,1,3)^T$. 
\end{ex}
Consequently, a straight-forward transposition of the bi-objective case is not possible.

%% file: multiobjective.tex
\section{Methods for the multi-objective case}\label{sec:litpobj}

Three methods have been proposed for the computation of the set of nondominated extreme points of MOIP. 

\subsection{The method by \citet{3SE}}\label{subsec:3SE}

The method by \citet{3SE} is based on the computation of a weight set decomposition. 
The weight set $W^0$ is defined by 
\begin{equation}\label{eq:W0} 
W^0 := \left \{\lambda \in \mR^p: \lambda_k >  0 \mbox{ for } k \in 
  \{1,\ldots,p\}, \lambda_p  =  1 - \sum_{k = 1}^{p-1}
  \lambda_k 
  \right\}, 
\end{equation} 
and can be seen as a normalized weight space. $W^0$ is a polytope of dimension $p-1$ and in particular, it is bi-dimensional in the three-objective case (very helpful for illustration purposes).

Given a supported nondominated point $y$, the set $W^0(y)$ is defined by
\begin{equation}\label{eq:W0y}
W^0(y) := \left\{ \lambda \in W^0 : \lambda^Ty = 
  \min\left\{ \lambda^Ty': y' \in \conv Y \right\} \right\}
\end{equation} 
and corresponds to the subset of weights $\lambda\in W^0$ for which $y$ is the image of an optimal solution of (MOIP$_\lambda$).
The method by \citet{3SE} is based on the following results.

\begin{prop}[\citealp{3SE}]\label{compW0(y)} Let $y$ be a supported nondominated point. 
\begin{enumerate}
\item $W^0(y) = \{\lambda  \in W^0 : 
 \lambda^Ty \leq \lambda^Ty' 
 \mbox{ for all } y' \in Y_{SN1} \setminus \{y\}\}.$  
\item $W^0(y)$ is a convex polytope.
\item Nondominated point $y$ is a nondominated extreme point of $\conv Y$ if
and only if $W^0(y)$ has dimension $p-1$. 
\item $W^0 = \bigcup_{y \in Y_{SN1}} W^0(y).$  
\item Let $S$ be a set of supported nondominated points. Then $$Y_{SN1}
\subseteq S \Longleftrightarrow W^0 = \bigcup_{y \in S} W^0(y).$$ 
\end{enumerate}
\end{prop}

\begin{Def}[\citealp{3SE}]\label{defadj}
Two nondominated extreme points $y^1$ and $y^2$ are called {\em adjacent} if
and only if their common facet $W^0(y^1) \cap W^0(y^2)$ in the weight set is a polytope of dimension $p - 2$.
\end{Def}

Given the set of nondominated extreme points $Y_{SN1}$, the decomposition of $W^0$ in sets $\{W^0(y)\; : \, y \in Y_{SN1}\}$ is given by Proposition \ref{compW0(y)}(4). 
Figure \ref{Decomp} illustrates this decomposition for the case of Example \ref{diff}.

\begin{figure}
\center
\includegraphics[width = 0.5\textwidth]{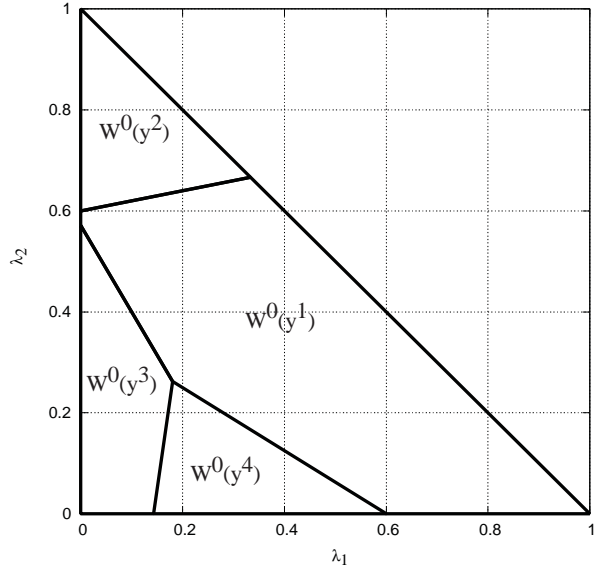}
\caption{Weight set decomposition for Example \ref{diff}}
\label{Decomp}
\end{figure}

Proposition \ref{compW0(y)}(1) provides a way to compute the sets $W^0(y)$ knowing the set $Y_{SN1}$. However, $Y_{SN1}$ is not known at the beginning of an iterative algorithm since it is the set that should be computed simultaneously with the weight set decomposition. 
\citet{3SE} have thus proposed to consider the directly computable sets 
$$W^0_S(y) := \{\lambda  \in W^0 : 
 \lambda^Ty \leq \lambda^Ty' 
 \mbox{ for all $y' \in S\setminus\{y\}$}\}$$
for subsets of supported nondominated points $S\subseteq Y_{SN}$.
Similar to the dichotomic scheme for the bi-objective case, a subset of supported nondominated points $S \subseteq Y_{SN}$ (containing possibly non-extreme points) is known at each step of the algorithm. 
Proposition \ref{compW0(y)}(5) is used as an optimality condition in the algorithm, in the sense that it allows to show that $Y_{SN1} \subseteq S$.
At any step of the algorithm, $W^0(y) \subseteq W^0_S(y)$ for all $y \in S$ and $\bigcup_{y \in S} W^0_S(y) = W^0$. Each new explored supported nondominated point can be used to update the sets $W^0_S(y)$. At termination of the algorithm, $W^0_S(y) = W^0(y)$ for all $y \in S$ and we have thus $Y_{SN1} \subseteq S$ by application of 
Proposition \ref{compW0(y)}(5). 

The principle of the method is to consider each known supported point $y \in S$, and to show either that $W^0_S(y) = W^0(y)$ or to identify new supported nondominated points allowing next an update of $W^0_S(y)$. To show that $W^0_S(y) = W^0(y)$ (or not) is done by showing that all facets of both polytopes are (not) common. 

More precisely, given two nondominated extreme points $y^1, y^2 \in S$ that are adjacent w.r.t.\ $S$, the question is whether the common facet $W^0_S(y^1) \cap W^0_S(y^2)$ of $W^0_S(y^1)$ and $W^0_S(y^2)$ is also the common facet $W^0(y^1) \cap W^0(y^2)$ of $W^0(y^1)$ and $W^0(y^2)$ or not. The facet $W^0_S(y^1) \cap W^0_S(y^2)$ is investigated in order to identify new supported points or to confirm that $W^0_S(y^1) \cap W^0_S(y^2) = W^0(y^1) \cap W^0(y^2)$. In particular if $p = 3$,  $W^0_S(y^1) \cap W^0_S(y^2)$ is an edge and its investigation is realized by the computation of the nondominated extreme points of a bi-objective problem. 
If $p \geq 4$, nondominated extreme points of $(p-1)$-objective problems must be computed, which is implemented by a recursive application of the method.


\subsection{The method by \citet{boekler15}}

Motivated by the recent work on a dual Benson's algorithm \citep[see][]{ehrg:adua:2012,hame:bens:2014}, \citet{boekler15} take a dual perspective on the weight set decomposition of \citet{3SE}. The following interpretation is based on these references and adapted to the notation from Section \ref{subsec:3SE}. Let
\begin{equation*}
\cl W^0  =
  \left \{\lambda \in \mR^p: \lambda_k \geq  0 \mbox{ for } k \in 
  \{1,\ldots,p\}, \lambda_p  =  1 - \sum_{k = 1}^{p-1}
  \lambda_k \right\} = \left\{ \lambda \in\mR^p_\geqq \,:\, \| \lambda\|_1 = 1 \right\}
\end{equation*} 
be the closure of $W^0$ (see \eqref{eq:W0}), that includes weights with components equal to $0$. For $\bar{\lambda}\in\cl W^0$, consider the linear programming relaxation of (MOIP$_{\bar{\lambda}}$) given by
$$ \min\{ \bar{\lambda}^T Cx \,:\, Ax=b,\, x\geqq 0\}
$$
and the corresponding dual linear program
$$ \max\{ b^T u\,:\, A^T u \leqq C^T \bar{\lambda} \}.
$$
Then we can associate with every weight $\bar{\lambda}:=(\lambda_1,\ldots,\lambda_p,1-\sum_{k=1}^{p-1}\lambda_k)^T\in \cl W^0$ (that is uniquely determined by its first $(p-1)$ components) a set of dual feasible solutions $\{ u\in\mR^m \,:\, A^T u \leqq C^T \bar{\lambda} \}$ with corresponding dual objective values $z=b^Tu$.
This information is comprised in the dual polyhedron 

\parbox{0.9\textwidth}{
\begin{eqnarray*}
\mathcal{D} := \left\{ (\lambda_1,\dots,\lambda_{p-1},z=b^Tu)^T\in\mR^p \right. & : &
\bar{\lambda}:=(\lambda_1,\ldots,\lambda_p,1-\sum_{k=1}^{p-1}\lambda_k)^T\in \cl W^0\\
&& \left. \text{~and~} A^T u \leqq C^T\bar{\lambda}\}
 \right\}
\end{eqnarray*}
}\parbox{0.065\textwidth}{\hfill (4)}

\noindent introduced in \citet{heyd:geom:2008}.
We are particularly interested in maximizing $z=b^Tu$ for different scalarizations $\bar{\lambda}$, that is, we are interested in the upper envelope of $\mathcal{D}$ w.r.t.\ the last component. In other words, we are looking for a maximal subset of $\mathcal{D}$ w.r.t.\ the cone
$K_p:=\{ (0,\ldots,0,z)^T\in\mR^p\,:\, z\geqq 0\}$. Using linear programming duality, this $K_p$-maximal subset of $\mathcal{D}$ can be written as

\parbox{0.9\textwidth}{
\begin{eqnarray*}
K_p\!-\!\max\, \mathcal{D} = \left\{ (\lambda_1,\dots,\lambda_{p-1},z)^T\in\mR^p \right. & : &
\bar{\lambda}:=(\lambda_1,\ldots,\lambda_{p-1},1-\sum_{k=1}^{p-1}\lambda_k)^T\in \cl W^0\\
&& \left. \text{~and~} z=\min\{\bar{\lambda}^Ty : y \in \conv Y\}
 \right\}.
\end{eqnarray*}
}\parbox{0.065\textwidth}{\hfill (5)}
where in this context, $Y = \conv Y = \{Cx \,:\, Ax=b,\, x\geqq 0\}$.

\noindent \citet{heyd:geom:2008} showed that the $K_p$-maximal facets of $\mathcal{D}-K_p$ correspond to the (weakly) nondominated extreme points of $(\conv Y) + \mR^p_{\geqq}$ (note that all extreme points of $(\conv Y)+\mR^p_{\geqq}$ are in fact nondominated), i.e., to the nondominated extreme points in $Y_{SN1}$, and vice versa. Note that this observation is also reflected in Proposition \ref{compW0(y)}(3), i.e., there is a one-to-one correspondence between the $K_p$-maximal facets of $\mathcal{D}-K_p$ and the sets $W^0(\bar{y})$ with $\bar{y}\in Y_{SN1}$. Note also that 
$\bar{y}$ remains optimal for all weights $\bar{\lambda}=(\lambda_1,\ldots,\lambda_p,1-\sum_{k=1}^{p-1}\lambda_k)^T\in \cl W^0(\bar{y})$ and thus $\bar{z}=\bar{\lambda}^T\bar{y}$ changes linearly with $\bar{\lambda}$ on 
$\cl W^0(\bar{y})$. Consequently, there is also a one-to-one correspondence between the extreme points of $\mathcal{D}-K_p$ and the extreme points of the closure of the weight sets $\cl W^0(\bar{y})$, $\bar{y}\in Y_{SN1}$. 

From an algorithmic point of view, this dual interpretation gives rise to an alternative way to (implicitly) compute the weight set decomposition.  Following the description in \citet{boekler15}, the dual method also works with a partial list $S$ of nondominated extreme points of $\conv Y$ that define a subset of the facet describing inequalities of $\mathcal{D}-K_p$.  
In order to test whether an extreme point $(\lambda_1,\dots,\lambda_{p-1},z)^T$ of the current intermediate dual polyhedron (given by the facet describing inequalities) is also an extreme point of $\mathcal{D}-K_p$, a weighted sum problem with weight $(\lambda_1,\ldots,\lambda_p,1-\sum_{k=1}^{p-1}\lambda_k)^T$ is solved. \citet{boekler15} suggest to apply a lexicographic weighted sum scalarization in order to avoid non-extreme supported efficient solutions and also dominated solutions in the case of weights with components equal to $0$. If the optimal objective value equals $z$, the above question is answered positively, and otherwise a new facet describing inequality is found and added to the intermediate description of $\mathcal{D}-K_p$. 

\citet{boekler15} noticed that this algorithm can also be applied to multi-objective combinatorial optimization problems. Indeed, for all combinatorial optimization problems (and more generally for all mixed-integer linear programmes), there exists a linear programming formulation, which could be used in the definition of the dual Polyhedron (4).  We emphasize that to know the linear formulation of the problem is in fact not necessary, since only optimal objective values of weighted sum problems are required in the method, and these optimal values do not depend on the formulation of the problem, as can be seen for example in formulation (5).

In comparison to the algorithm of \citet{3SE}, the dual method avoids the time-consuming analysis of common facets of weight sets between pairs of nondominated extreme points, c.f.\ Section~\ref{subsec:3SE}. 

Under the assumptions that the number of objective functions $p$ is fixed, the weighted sum scalarizations (MOIP$_\lambda$) can be solved in polynomial time, and the polytope $\cal D$ is computed with a dual algorithm to a statical convex hull algorithm \citep[e.g.][]{Chazelle93}, \citet{boekler15} show that the dual Benson method runs in output polynomial time, i.e. its running time is bounded by a polynomial in the input and the output size. Moreover, if the lexicographic weighted sum problems can still be solved in polynomial time, the dual Benson method runs in incremental polynomial time, i.e., the $k$-th delay (the running time between the output of the $k$-th and the $(k+1)$-st solution) is bounded by a polynomial in the input and $k$.

\subsection{The method by \citet{MSOzpKok}}

The method by \citet{MSOzpKok} is performed directly in the objective space. In order to compute all nondominated extreme points, \citet{MSOzpKok} consider the convex hull $\conv Y$ of the outcome set $Y$, and in particular $(\conv Y)_N$ the {\em nondominated frontier}. As $\conv Y$ is a polytope,  \citet{MSOzpKok} define the nondominated frontier as the union of all of its nondominated faces. However, if in the bi-objective case all maximal nondominated faces are nondominated edges, with $p$ objectives there can be maximal nondominated faces of dimension 1, 2,$\ldots$, $p - 1$. \citet{MSOzpKok} illustrate this fact in the three-objective case using the numerical instance of Example \ref{diff}. There is indeed one maximal nondominated face of dimension 2 (facet) given by the weight defined by the intersection of the sets $W^0(y^1)$, $W^0(y^3)$ and $W^0(y^4)$, and one maximal nondominated face of dimension 1 given by any weight in the interior of the edge $W^0(y^1)\cap W^0(y^2)$ (see Figure \ref{Decomp}).  

\citet{MSOzpKok} introduce {\em dummy points} in order to modify the structure of the nondominated frontier. These points are defined by 
$$m^q := Me_q \hbox{ for $q = 1,\ldots,p$}$$
where $e_q$ is the $q$-th unit vector, and $M$ is a large positive constant. The set of dummy points $Y_M$ is defined by $\bigcup_{q = 1}^p \{m^q\}$. As \citet{MSOzpKok} make the assumption that $z_k(x) > 0$ for all $k \in \{1,\dots,p\}$, dummy points are thus nondominated points of $Y \cup Y_M$. $M$ must be chosen large enough so that nondominated extreme points of $\conv Y$ remain nondominated extreme points of $\conv (Y \cup Y_M)$. Lower bound values for $M$ that guarantee this, under the assumption that objective coefficients and variables are integer, are given in \citep{OZ08}. The following properties are then verified.
\begin{prop}[\citealp{MSOzpKok}]\label{objView}
The introduction of dummy points in the set of nondominated extreme points 
has the following consequences.
	\begin{enumerate}
	\item For $y \in Y_{SN1} \cup Y_M$, we denote by $W^{0M}(y)$ the set of weights $\lambda \in W^0$ for which $y$ is the image of an optimal solution of $(MOIP_\lambda)$, modified by the introduction of dummy points. The weight set decomposition becomes $W^0 = \bigcup_{y \in Y_M\cup Y_{SN1}}W^{0M}(y)$, and $W^{0M}(y)\cap\fr(W^0) = \emptyset$ for all $y \in Y_{SN1}$. 
	\item Every point $y \in Y_{SN1}$ is adjacent to at least $p$ points in $Y_{SN1} \cup Y_M$.
	\item Any pair of dummy points $m^q$, $m^r \in Y_M$ are adjacent for $p\geq 3$.
	\item If $Y_{SN1} \neq \emptyset$ and $p \geq 3$, then every point $y \in Y_{SN1} \cup Y_M$ is adjacent to at least $p$ points in $Y_{SN1} \cup Y_M$.
	\item All maximal nondominated faces of $\conv (Y \cup Y_M)$ are facets. Thus, $(\conv(Y_{SN1} \cup Y_M))_N$ is fully described by a union of facets. 
	\end{enumerate}
\end{prop}

The method by \citet{MSOzpKok} consists in the computation of $\conv(Y_{SN1} \cup Y_M)_N$, the extreme points of which are $Y_{SN1} \cup Y_M$. According to Proposition \ref{objView}(5), all maximal nondominated faces are facets. It is therefore possible to design an algorithm that identifies these facets.

The algorithm starts with the computation of $Y_M$. A set $S \subseteq Y_{SN}$ is iteratively computed until all facets of $\conv(Y_{SN1} \cup Y_M)_N$ are identified. 

At each iteration of the algorithm, a subset of $p$ points $\{r^1,\ldots,r^p\} \subset S \cup Y_M$ (initially $\{m^1,\ldots,m^p\}$) called {\em stage}, is considered with the normal $\lambda \in \mR^p$ of the hyperplane it defines. If $\lambda \in \mR^p_>$, then the solution of (MOIP$_\lambda$) allows to obtain a supported nondominated point $y$. If $\lambda^Ty = \lambda^Tr^1$ then $\conv\{r^1,\ldots,r^p\}$ is either a facet or a part of a facet of $\conv(Y_{SN1} \cup Y_M)_N$. Otherwise, $y$ is added to $S$ and new stages are generated $\{y, r^2,\ldots,r^p\}$, $\{r^1,y,r^3,\ldots,r^p\}$,..., $\{r^1,\ldots,r^{p-1},y\}$ and considered next if not yet visited.

However, despite Proposition \ref{objView}(5), an arbitrary stage $\{r^1,\ldots,r^p\} \subset S \cup Y_M$ does not necessarily have a normal $\lambda \in \mR^p_>$ (note that in this case, $\conv\{r^1,\ldots,r^p\}$ is not part of $\conv(S \cup Y_M)$). In this case, the weighted sum problem (MOIP$_\lambda$) is not solved. Nevertheless, to observe such a stage cannot be a stopping condition for the algorithm, as the enumeration would be incomplete (see Example \ref{diff}). In order to continue the enumeration, the authors use implicitly the fact that $(\conv(S \cup Y_M))_N$ is a union of facets. Consequently, other (unvisited) stages can be defined using points in $S \cup Y_M$. 
\citet{MSOzpKok} have proposed to choose a point $y \in (Y_{M} \cup S)\setminus\{{r^1},\ldots,{r^p}\}$ together with $\lambda' \in \mR^p_>$ such that $\lambda'^Ty \leq \lambda'^Tr^{i}$ for all $i \in \{1,\ldots,p\}$. Such a point $y$ necessarily exists since $W^0_S(y) \neq \emptyset$ for all $y \in (Y_{M} \cup S)\setminus\{{r^1},\ldots,{r^p}\}$. The stages $\{y, {r^2},\ldots,{r^p}\}$, $\{{r^1},y,{r^3},\ldots,{r^p}\}$, $\ldots$, $\{{r^1},\ldots,{r^{p-1}},y\}$ are next generated to continue the execution of the algorithm.

%% file: approximation.tex
\section{Approximation Algorithms}\label{sec:litapprox}

There is an extensive literature on the approximation (with a guarantee of quality) of the set of nondominated points of convex or non-convex multi-objective problems \citep[see][for a review]{RW05}. 
We are interested here only in methods that generate convex approximations of multi-objective problems.

The quality of the approximation is usually measured according to an indicator that can be very different for different methods. 
 An a priori quality can be defined by fixing a target value for the quality indicator. In particular by fixing the quality indicator to 0, the obtained approximation becomes exact, i.e.  all nondominated extreme points are found. Even if it is not the purpose of these methods, this shows that methods able to compute the set of nondominated extreme points of any MOIP have been proposed before the methods proposed in \citep{3SE}  and \citep{MSOzpKok}. Of course, we can expect some weaknesses in the use of approximation methods used with an exact purpose, due to their initial design for another purpose.  Our aim here is not to give a complete overview of these methods. We will just review properties that will be useful for the (exact) method we propose next, or that can be seen as related to our work.

\subsection{The method by \citet{SKW02}}

\citet{SKW02} have proposed methods for the approximation of the nondominated set of a multi-objective programme.
For convex problems, their method can be interpreted as an extension of the dichotomic search to higher dimensions: starting from an initial approximation given by the lexicographic minima and using the nadir point (or an approximation of the nadir point) as reference point $y^0$, the convex hull of all these points is computed. In each iteration of the procedure, a cone spanned by $y^0$ and a facet of this convex hull 
is selected for further refinement based on a problem specific error measure.
In this cone, the normal of the defining facet is used to define weights $\lambda$ for a 
subproblem (MOIP$_\lambda$). The solution of this subproblem (MOIP$_\lambda$) leads to a new point that is included in the convex hull for the next iteration.
The procedure stops as soon as the approximation error falls below a prespecified threshold.
In this way, a polyhedral approximation of the set $\conv Y \cap (y^0 - \mR^p_\geqq)$ is generated.

Note that if this method is applied in order to generate all nondominated extreme points of (MOIP),
the reference point has to be selected to satisfy $y^0 \geqq y^N$ to ensure that $Y_{SN} \subseteq Y_N \subseteq \conv Y \cap (y^0 - \mR^p_\geqq)$.

The difficulties indicated in Section \ref{subsec:difficulties} are also noted in \citet{SKW02}.
Nevertheless, weights $\lambda \not \in \mR^p_>$ are used in this procedure to generate points that are not necessarily nondominated, but useful for the update of the convex hull. This method is thus not stopped because of the presence of weights with negative components. Applying this idea in the case of Example \ref{diff}, we could keep the point $y^d = (16,20,16)^T$ obtained by the solution of (MOIP$_{(1,-40,-28)}$) and compute  $\conv\{y^1,y^2,y^3,y^d\}$ in order to obtain new facets with normals that can be used to define new weights for later iterations. Finally, the whole set $\conv Y$ or a subset of it, can be computed in this way and $(\conv Y)_N$ can be deduced by a filtering step.

\subsection{The method by \citet{RenDamHer11}}

\citet{RenDamHer11} develop another method to approximate the nondominated set of
multi-objective programmes with convex objective functions and feasible sets.
Their method is also based on the computation 
of the convex hull of a growing set of points $S$.
Interestingly enough, they add dummy points to the set of points $S$
that constitutes the current approximation of the nondominated set.
The aim of this is to partially cope with the difficulties related to 
dominated facets.
Namely, for each extreme point $y$ of $S$, $p$ dummy points $d^1(y), \dots, d^p(y)$ 
are defined in the following way:
$$d^i_k (y) = \begin{cases}
                p y_k^{ub} + \theta & \hbox{if $k = i$}\\
                y_k & \hbox{otherwise}
              \end{cases}$$
for all $i, k \in \{1, \dots, p\}$, where $y_k^{ub}$ 
is an upper bound on the $k$th component value of any nondominated point and $\theta$ 
is a positive constant.

It is then shown \citep[Lemma 2]{RenDamHer11} 
that all \emph{relevant} facets, i.e. facets that have at least one non-dummy point as extreme point,
are weakly nondominated in the sense that any normal vector to any such facet pointing to the interior of $\conv S$ is in $\mR^p_{\geq}$.

Compared to \citet{MSOzpKok}, 
this approach defines a large number of dummy points (even if some are redundant and therefore filtered)
that increase the numbers of extreme points and facets of $\conv S$.

%% file: newmethod.tex
\section{A new Exact Method}\label{newMethod}

In this section, we extend some properties proposed in \citep{3SE}, \citep{MSOzpKok} and \citep{SKW02}, and we propose next a new method.

\subsection{Further analysis}\label{furana}

We suppose that we know a subset $S$ of supported nondominated points, and we assume that the non-extreme points are filtered from $S$ as soon as possible. If $|S| > p$ then $\conv S$ is a full-dimensional polytope. The facets of this polytope can be computed and updated (when points are added to $S$) using 
a convex hull algorithm. We consider the sets $W^0(y^i)$ and $W^0_S(y^i)$ for all $y^i \in S$ as defined in \citep{3SE}. 
In the following, it will be important to notice that $W^0$ is an open polytope, and for any supported point $y^i$, $W^0(y^i)$ and $W^0_S(y^i)$ may be closed polytopes, or polytopes that are neither closed nor open. These polytopes may have open vertices and open faces. These open faces and open vertices are located on the boundary $\fr(W^0)$ of $W^0$. 

Knowing a weight set decomposition, Proposition~\ref{faces} below gives a characterization of the nondominated facets of the convex hull of a subset of supported points. This result has been stated with a different formulation and a different proof in \citep{3SE} and \citep{TheseAnthony}.

\begin{prop}\label{faces}
Let $S$ be a set of supported nondominated points, then there is a one-to-one correspondence 
between weights given by extreme points of $W^0_S(y)$ 
that are located in $W^0$, 
for $y \in S$,
and weights associated to facets of $(\conv S)_N$.
\end{prop}
\begin{proof}
For all $y \in S$, $W^0_S(y)$ can be described by a minimal set of (non-redundant) constraints 
$\lambda^Ty \leq \lambda^Ty^i$ for all $y^i \in I$ 
where $I \subseteq S \setminus \{y\}$ (plus possibly some of the constraints $\lambda_k > 0$, $k = 1,\ldots,p$ and $\lambda_p = 1 - \sum_{k = 1}^{p-1} \lambda_k$). 
Given an extreme point $\hat\lambda$ of $W^0_S(y)$ located in $W^0$,
$\hat\lambda$ {satisfies} at least $p - 1$ of these constraints with equality 
as $W^0_S(y)$ is a polytope of dimension $p-1$. 
Therefore, {there is a subset $K\subset I$ with $|K| \geq p-1$ such that} $\hat\lambda^T y = \hat\lambda^T y^i$ for all $y^i \in K$. 
As $\hat\lambda \in W^0_S(y)$, for all $\hat y \in S \setminus \{y\}$ 
we have $\hat\lambda^T y \leq \hat\lambda^T\hat y$, 
and thus for all $y^i \in K$, $\hat\lambda^Ty^i \leq \hat\lambda^T\hat y$. 
In other words, $\hat\lambda \in W^0_S(y) \cap \bigcap_{y^i \in K} W^0_S(y^i)$ 
and $F =  \conv (K\cup\{y\})$ is a part of a facet of $\conv S$.
Since $\hat\lambda$ is a normal to $F$ pointing to the interior of $\conv S$
and  $\hat \lambda \in W^0 \subset \mR^p_>$, $F$ is a part of a nondominated facet of $\conv S$. 

Conversely, given a nondominated facet $F$ of $\conv S$, 
it is immediate that its associated weight belongs to $\mR^p_>$ 
and that there is thus a positive multiple in $W^0$.
Moreover, this weight is necessarily located 
at the intersection of the sets $W^0(y^i)$ 
where the $y^i$'s are the extreme points of $F$.
\end{proof}\medskip

The weights associated to other (i.e. dominated) facets of $\conv S$ cannot be defined using a decomposition of $W^0$ as such  weights do not belong to $\mR^p_>$. However, weights with negative components are considered in \citep{SKW02}, in order to avoid a premature termination of the dichotomic scheme, i.e., to find appropriate weights $\lambda \in \mR^p_>$ in later iterations. In the following, we analyze under what conditions negative components in a weight associated to $p$ points in $S$ occur.

\begin{ex}\label{exExt}
We consider the instance of the Assignment Problem given in Example~\ref{diff}. Suppose we start again by computing the set of solutions that minimize each objective. Computing $\conv S$ (not full-dimensional as $|S| = 3$), we get a single facet with a corresponding weight $\lambda = (-1,40,28)^T$ that is not suitable to continue the classical dichotomic algorithm. We analyze the cause for this using the weight space.

We consider the computation of the sets $W^0_S(y^i)$ 
and we obtain the decomposition of $W^0$ given by Figure~\ref{decompWp}.
We can note that the extreme points of the polytopes $\cl(W^0_S(y^i))$ 
are not located in $W^0$ and that consequently 
there is no facet in $(\conv S)_N$. 
In other words, all facets of $\conv S$ are dominated. 
To find the weight associated to such a facet of $\conv S$,
we must extend the facets (here edges) of the sets $W^0_S(y^i)$ outside of $W^0$. 
We obtain the weight $\lambda = (-\frac{1}{67},\frac{40}{67},\frac{28}{67})$ 
as an intersection of the extensions of $W^0_S(y^1)$, $W^0_S(y^2)$ and $W^0_S(y^3)$ 
(see Figure~\ref{decompExt}). 
\begin{figure}[htb] 
\begin{center}
  \begin{minipage}[t]{0.45\textwidth}
  \includegraphics[width=0.96\textwidth]{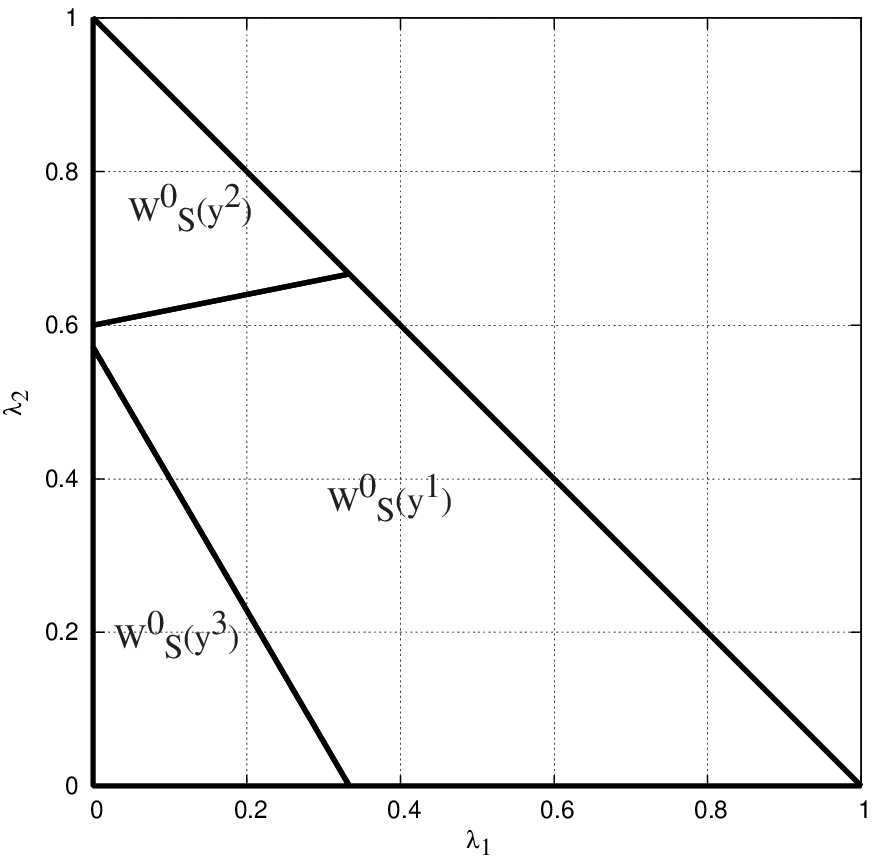}
  \caption{Decomposition of $W^0$ with $W^0_S(y^1)$, $W^0_S(y^2)$ and $W^0_S(y^3)$.}
  \label{decompWp} 
  \end{minipage}
  \hfill
  \begin{minipage}[t]{0.45\textwidth}
  \includegraphics[width=\textwidth]{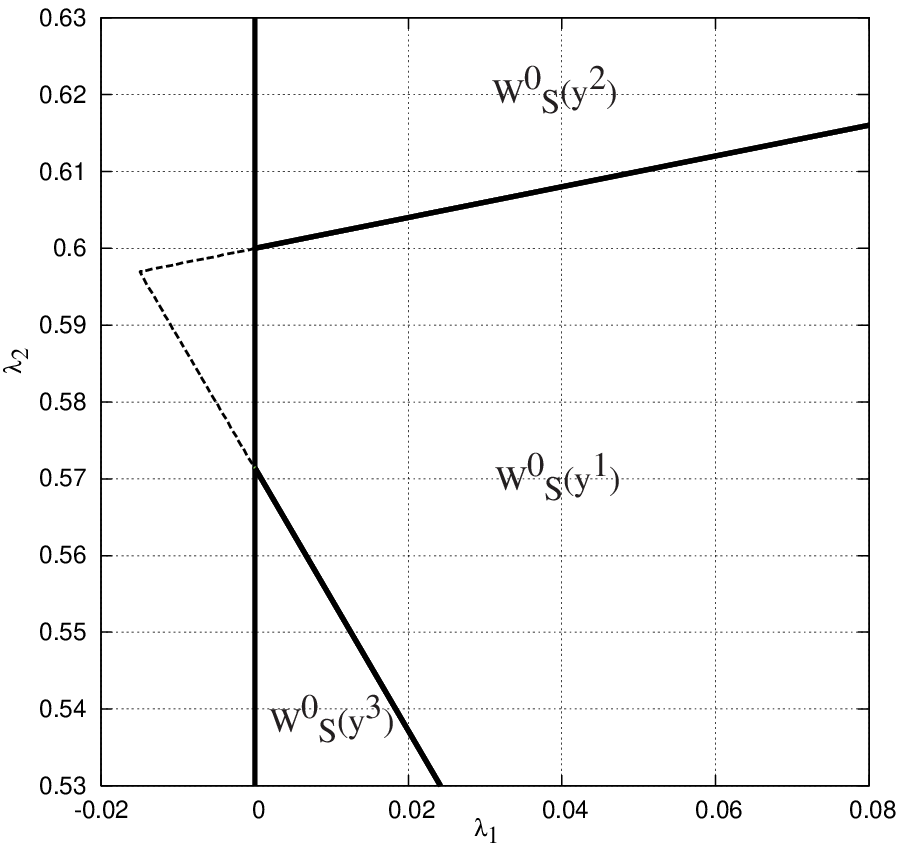}
  \caption{Extension of the decomposition outside of $W^0$.} 
  \label{decompExt}
  \end{minipage}
\end{center}
\end{figure}
\end{ex}

In order to analyse the observation of Example~\ref{exExt}, we relax the constraints of strict positivity of the components of weights in the definition of $W^0$ and we obtain the extended weight set  
$$W^0_{\ext} := \{\lambda \in \mR^p \hbox{ with $\sum_{k = 1}^p \lambda_k = 1$}\}.$$
It should be noted that if $W^0_{\ext}$ includes weights with negative components, it does not include weights $\lambda$ such that $\sum_{k = 1}^p \lambda_k = 0$ or  $\sum_{k = 1}^p \lambda_k = -1$. We obviously have $W^0 \subset W^0_{\ext}$. We {accordingly} extend the definition of the sets $W^0_S(y)$ to $W^0_{S_{\ext}}(y)$ for all $y \in S$ by defining
\begin{align*}W^0_{\ext}(y) &:=  \{\lambda  \in W^0_{\ext} :  \lambda^Ty \leq \lambda^Ty'  \mbox{ for all }y' \in Y_{SN1}\setminus\{y\}\}\\
W^0_{S_{\ext}}(y) &:=  \{\lambda  \in W^0_{\ext} :  \lambda^Ty \leq \lambda^Ty'  \mbox{ for all }y' \in S\setminus\{y\}\}.
\end{align*}
Proposition~\ref{W0ext} provides the main properties of these new sets.

\begin{prop}\label{W0ext}
Let $y \in S$ be a supported nondominated point.
\begin{enumerate}
\item $W^0_{\ext}$ is a hyperplane in $\mR^p$.
\item For all $y \in S$, $W^0_{S_{\ext}}(y)$ is a closed polyhedron.
\item If $\cl(W^0_S(y))$ has no extreme point in $\fr(W^0)$, then $W^0_S(y) = W^0_{S_{\ext}}(y)$. 
\item If $\cl(W^0_S(y))$ has {an extreme point} in $\fr(W^0)$, then $W^0_S(y) \subset W^0_{S_{\ext}}(y)$.
All extreme points of $W^0_S(y)$ that are located in $W^0$ are also extreme points of $W^0_{S_{\ext}}(y)$. The other extreme points of $W^0_{S_{\ext}}(y)$ (if any) are located either in $\fr(W^0)$ or outside of $\cl(W^0)$.
\item Extreme points of $W^0_{S_{\ext}}(y)$ correspond to weights associated to some facets of $\conv S$.
\end{enumerate}
\end{prop} 

\begin{proof}
\begin{enumerate}
\item $\sum_{k = 1}^p \lambda_k = 1$ is the equation of a hyperplane in $\mR^p$.

\item For any $y \in S$, $W^0_{S_{\ext}}(y)$ is defined by a finite set of linear constraints, and it is therefore a closed polyhedron. {Note that in general this does not imply that $W^0_{S_{\ext}}(y)$ is also bounded.}

\item $W^0_S(y)$ and $W^0_{S_{\ext}}(y)$ are defined using the same linear constraints, the only difference is that $\lambda_k > 0$ for all $k \in \{1,\ldots,p\}$ in the definition of $W^0_S(y)$. If $\cl(W^0_S(y))$ has no point in $\fr(W^0)$ then the constraints $\lambda_k > 0$ for all $k \in \{1,\ldots,p\}$ are redundant in its definition. We have therefore $W^0_S(y) = W^0_{S_{\ext}}(y)$.

\item 
If $\cl(W^0_S(y))$ has {an extreme point} in $\fr(W^0)$ then $W^0_S(y) \neq \cl(W^0_S(y))$. Moreover, we {immediately have} $\cl(W^0_S(y)) \subseteq  W^0_{S_{\ext}}(y)$ by definition of both sets. Consequently, $W^0_S(y) \subset W^0_{S_{\ext}}(y)$.
For all extreme points of $\cl(W^0_S(y))$ that do not belong to $\fr(W^0)$, 
none of the constraints $\lambda_k > 0$ for $k \in \{1,\ldots,p\}$ are active. Consequently, these extreme points are common to $W^0_S(y)$ and $W^0_{S_{\ext}}(y)$. Finally as $W^0_S(y) \subset W^0_{S_{\ext}}(y)$, extreme points located in $\fr(W^0)$ or outside of $\cl(W^0)$ and/or extreme rays may be necessary to complete the definition of $W^0_{S_{\ext}}(y)$.
\item The proof is identical to the first part of the proof of Proposition \ref{faces}, except that the weights defined by extreme points of $W^0_{S_{\ext}}(y)$ are not necessarily located in $W^0$.
\end{enumerate}
\end{proof}\medskip

Proposition~\ref{faces2} can be seen as an extension of Proposition~\ref{faces} providing a link between $(\conv S)_N$, $(\conv S)_{wN}$ and the sets $W^0_{S_{\ext}}(y)$ for $y \in S$.

\begin{prop}\label{faces2}
Let $S$ be a set of supported nondominated points.
\begin{enumerate}
\item There is a one-to-one correspondence between weights given by extreme points of $W^0_{S_{\ext}}(y)$ that are located in $W^0$, for $y \in S$, and weights associated to facets of $(\conv S)_N$. 
\item There is a one-to-one correspondence between weights given by extreme points of $W^0_{S_{\ext}}(y)$ that are located in $\cl(W^0)$, for $y \in S$, and weights associated to weakly nondominated facets of $\conv S$. 
\end{enumerate}
\end{prop}

\begin{proof}
Proposition~\ref{W0ext}~$(3)$ and $(4)$ implies that for a given $y \in S$, either $W^0_{S_{\ext}}(y) = W^0_S(y)$ or $W^0_S(y) \subset W^0_{S_{\ext}}(y)$. In particular, all extreme points of $W^0_S(y)$ that are located in $W^0$ are also extreme points of $W^0_{S_{\ext}}(y)$. 
\begin{enumerate}
\item The statement is next a direct consequence of Proposition~\ref{faces}.
\item Given an extreme point $\lambda$ of $W^0_{S_{\ext}}(y)$ located in $\fr(W^0)$, Proposition~\ref{W0ext}(5) implies that it corresponds to the weight associated to a facet. Moreover, as $\lambda \in \fr(W^0)$, this facet is weakly nondominated. Conversely, given a weakly nondominated facet $F$ of $\conv S$, its associated weight belongs to $\mR^p_\geq$, and there is a positive multiple in $\fr(W^0)$. Moreover, this weight is necessarily located in the intersection of the sets $W^0_{S_{\ext}}(y^i)$ where the $y^i$'s are the extreme points of $F$.
\end{enumerate}
\end{proof}\medskip

We cannot expect a one-to-one correspondence between the facets of $\conv S$ and extreme points of the sets $W^0_{S_{\ext}}(y)$ for $y \in S$. Indeed, dominated facets may also have a normal $\lambda$ such that $\sum_{k = 1}^p \lambda_k = 0$ or  $\sum_{k = 1}^p \lambda_k = -1$, and such weights are not considered in the extended weight set $W^0_{S_{\ext}}$. 

Finally, our interest is not generally in the dominated facets of $\conv S$, but in facets for which at least one extreme point is nondominated. Proposition~\ref{MainResult} will be the main result to propose new extensions of the dichotomic scheme to the multi-objective case.

\begin{prop}\label{MainResult}
Let $y \in S$, $W^0_S(y)$ and $W^0_{S_{\ext}}(y)$ be its associated weight sets. 
\begin{enumerate}
\item If $W^0_{S_{\ext}}(y) \subset W^0$, then every facet of $\conv S$ such that $y$ is an extreme point of this facet is nondominated. 
\item If $W^0_{S_{\ext}}(y) \subset \cl(W^0)$, then every facet of $\conv S$ such that $y$ is an extreme point of this facet is at least weakly nondominated. 
\end{enumerate}
\end{prop}

\begin{proof}
\begin{enumerate}
\item Suppose that $y$ is an extreme point of a dominated facet $F$, then any normal $\lambda$ to $F$ pointing to the interior of $\conv S$ has at least one negative or zero component. If $W^0_{S_{\ext}}(y) \subset W^0$ then by definition there is no weight $\lambda \in W^0_{\ext}$ (i.e., {satisfying}  $\sum_{k=1}^p \lambda_k = 1$) defined by simultaneously positive and negative/zero components such that $\lambda^Ty \leq \lambda^Ty'$ for all $y' \in S$. Hence, any such normal $\lambda$ to $F$ pointing to the interior of $\conv S$ should {satisfy} either $\sum_{k = 1}^p \lambda_k = 0$ or $\sum_{k = 1}^p \lambda_k = -1$ (with a possible multiplication of the weight by a positive constant). This implies in both cases, negative components in the weight $\lambda$.

We suppose first that a normal of $F$ denoted $\lambda^-$ and  pointing to the interior of $\conv S$ {satisfies} $\sum_{k = 1}^p \lambda^-_k = -1$. Since $y$ is an extreme point of $F$, we have therefore $(\lambda^-)^Ty \leq (\lambda^-)^Ty'$ for all $y' \in S$. Let $\lambda^+$ be a weight in $W^0_S(y)$ then $\lambda^+ \neq -\lambda^-$, otherwise we have simultaneously $(\lambda^+)^Ty \leq (\lambda^+)^Ty'$ and $(- \lambda^+)^Ty \leq (- \lambda^+)^Ty'$ for all $y' \in S$, which is not possible if $\conv S$ is full-dimensional. Next, the weight $\lambda^\alpha := \alpha\lambda^+ + (1 - \alpha)\lambda^-$ with $\alpha \in [0,1]$, {satisfies} $(\lambda^\alpha)^Ty \leq (\lambda^\alpha)^Ty'$ for all $y' \in S$. In particular, if $\alpha = \frac{1}{2}$, we have $\sum_{k=1}^p \lambda^\alpha_k = 0$ and $\lambda^\alpha$ has necessarily simultaneously positive and negative components. Moreover, there exists $\epsilon > 0$ such that if $\alpha = \frac{1}{2} + \epsilon$, $\lambda^\alpha$ still has simultaneously positive and negative components and $\sum_{k=1}^p \lambda^\alpha_k > 0$. It is therefore possible to multiply $\lambda^\alpha$ by a positive constant in order to obtain a weight $\lambda^\alpha$ such that $\sum_{k=1}^p \lambda^\alpha_k = 1$. In other words, we obtain a weight in $W^0_{S_{\ext}}(y) \setminus W^0$ which contradicts the assumption that $W^0_{S_{\ext}}(y) \subset W^0$.

We suppose now that a normal of $F$ denoted by $\lambda^0$ and pointing to the interior of $\conv S$ {satisfies} $\sum_{k = 1}^p \lambda^0_k = 0$. Necessarily, $\lambda^0$ has simultaneously positive and negative components. Using the same idea as above, we obtain directly a contradiction with the assumption that $W^0_{S_{\ext}}(y) \subset W^0$. 

Finally, the normals to the only facets such that $y$ is an extreme point are defined by the extreme points of $W^0_{S_{\ext}}(y)$. As these extreme points are located in $W^0$,  Proposition~\ref{faces2}(1) implies that these facets are nondominated.
\item We can show analogously that the normals to the only facets such that $y$ is an extreme point are defined by the extreme points of $W^0_{S_{\ext}}(y)$.
As these extreme points are located in $\cl(W^0)$, Proposition~\ref{faces2} implies that these facets are weakly nondominated (normal defined by a weight in $\fr(W^0)$) or nondominated (normal defined by a weight in $W^0$).
\end{enumerate}

\end{proof}\medskip

\subsection{Algorithms}

In the following, we use the properties stated in Subsection~\ref{furana} in order to propose two new algorithms. Both algorithms are evolutions of Algorithm~\ref{Balloon}, using a particular initialization, and reducing the number of iterations. 
Algorithm~\ref{Balloon} can be seen as an exact version of the method proposed by \citet{SKW02}. 
No order is used here to select a facet (as all must be explored) 
and nothing is done to avoid the exploration of dominated facets. 
It ensures the computation of all extreme points of $\conv Y$, 
and all nondominated extreme points can next be deduced.


\begin{algorithm}[h]
\caption{Algorithm {\tt Inflate\_Balloon}}
\label{Balloon}
\begin{algorithmic}[1]
\REQUIRE An instance of an MOILP
\ENSURE $S$: the set of all extreme points of $\conv Y$
\STATE  Let $S$ be a set of known points on the boundary of $\conv Y$ such that $\conv S$ is full-dimensional 
\REPEAT
\STATE Compute/Update $\conv S$
\STATE Choose any unexplored facet $F$ to obtain a weight $\lambda$ defined by its normal
\STATE Solve the weighted sum problem with objective function $\lambda^Tz$ (even if $\lambda \not\in \mR^p_>$)
\STATE Let $x$ be the optimal solution, $y := z(x)$ and $y'$ be a point of $F$
\IF{$\lambda^Ty' = \lambda^Ty$}
\STATE A (part of a) facet of $\conv Y$ is identified
\ELSE  
\STATE A new point is found and added to $S$
\ENDIF
\UNTIL{All facets of $\conv S$ have been explored}
\end{algorithmic}
\end{algorithm}

We first consider the use of dummy points as proposed in \citep{MSOzpKok}. The introduction of these points in the list of supported nondominated points has been proposed to modify the structure of the convex hull of nondominated extreme points. 
In the following, we analyse the consequence for the sets $W^0(y)$, $W^0_S(y)$ and $W^0_{S_{\ext}}(y)$
defined with respect to $S \cup Y_M$, for $y \in S$. In order to avoid a confusion with/without the potential use of dummy points, we will use the modified notations if dummy points are considered in the weight set decomposition: $W^{0M}(y)$, $W^{0M}_S(y)$ and $W^{0M}_{S_{\ext}}(y)$. These are just notations and the statements of Propositions~\ref{faces}-\ref{MainResult} remain valid.

\begin{ex}\label{exdumm}
We consider again the instance of the Assignment Problem of Example~\ref{diff}. 
We suppose again that $S = \{y^1,y^2,y^3\}$ contains the three points minimizing each objective.
In Figure~\ref{Decomp4}, the weight set $W^0$ is decomposed into the sets $W^{0M}_S(y^1)$, $W^{0M}_S(y^2)$, $W^{0M}_S(y^3)$ 
and also the sets $W^{0M}_S(m^1)$, $W^{0M}_S(m^2)$, $W^{0M}_S(m^3)$. 
As we can see in Figure~\ref{Decomp4}, 
due to the dummy points $m^i$, 
the sets $W^{0M}_{S_{\ext}}(y^i)$ are included in $W^0$. 
Consequently, Proposition~\ref{MainResult} implies that all facets of $\conv (S \cup Y_M)$ 
such that at least one of its extreme points belongs to $S$ are nondominated. 
As the sets $\cl(W^{0M}_S(m^i))$ have extreme points in $\fr(W^0)$, 
Proposition~\ref{W0ext}$(4)$ implies that $W^{0M}_S(m^i) \subset W^{0M}_{S_{\ext}}(m^i)$ 
for all $m^i$. 
It is easy to see in Figure~\ref{Decomp4} that the sets $W^{0M}_{S_{\ext}}(m^i)$ 
are unbounded and have no extreme point outside of $W^0$. 
\end{ex}

\begin{figure}
\center
\includegraphics[width = 0.5\textwidth]{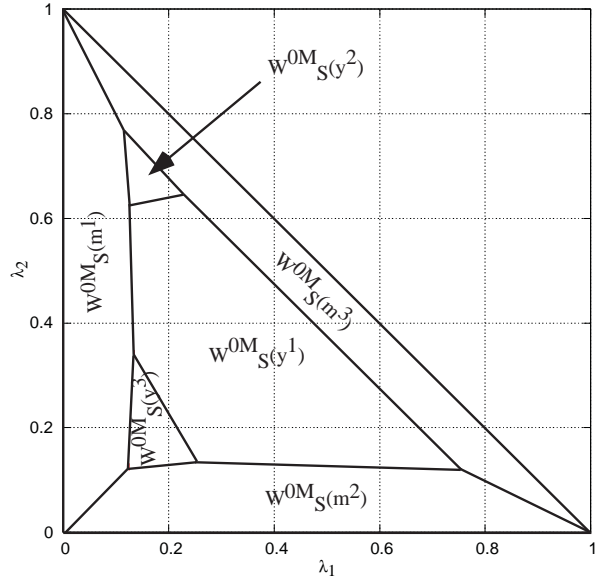}
\caption{Weight set decomposition for Example~\ref{exdumm} introducing dummy points.}
\label{Decomp4}
\end{figure}


Knowing an (extended) weight set decomposition, Proposition~\ref{PropDumm} below provides an explanation about the observations of Example~\ref{exdumm}. This is an extension of Proposition~\ref{objView}(5) proposed by \citep{MSOzpKok}. 

\begin{Lem}\label{W0Boundary}
Let $S$ be a subset of supported nondominated points and let $Y_M$ be the set of dummy points, then for all $y \in S$, $W^{0M}_{S_{\ext}}(y) \subset W^0$.
\end{Lem}

\begin{proof}
To show this statement, we just show that for any $y \in S$, $W^{0M}_{S_{\ext}}(y)$ has no intersection with $\fr(W^0)$, once we insert the dummy points in the weight set decomposition. Indeed, for any supported nondominated point $y \in S$, $W^{0M}_{S_{\ext}}(y) \cap W^0 \neq \emptyset$ and moreover $W^{0M}_{S_{\ext}}(y)$ is a polyhedron that is thus connected.

Let $y \in S$ and $\lambda \in W^{0M}_{S_{\ext}}(y)$, then we have $\lambda^Ty \leq \lambda^Ty'$ for all $y' \in S \cup Y_M$. In particular, we have $\lambda^Ty \leq \lambda^Tm^i$ for all $i \in \{1,\ldots,p\}$. Suppose that $\lambda \in \fr(W^0)$, then we have $\lambda \in \cl(W^0)$ with $\lambda_k = 0$ for at least one $k \in \{1,\ldots, p\}$. The inequality $\lambda^Ty \leq \lambda^Tm^k$ becomes equivalent to $\lambda^Ty \leq 0$. As $y$ has only positive components and $\lambda \in \cl(W^0)$, we necessarily have $\lambda^Ty > 0$ which is a contradiction to $\lambda^Ty \leq 0$. 
\end{proof}\medskip

\begin{prop}\label{PropDumm}
Let $S$ be a subset of supported nondominated points and let $Y_M$ be the set of dummy points.
All facets of $\conv (S \cup Y_M)$ are nondominated, except the one defined by $\{m^1,\ldots,m^p\}$.
\end{prop}

\begin{proof}
According to Lemma~\ref{W0Boundary}, for all $y \in S$, $W^{0M}_{S_{\ext}}(y) \subset W^0$. Hence, according to Proposition~\ref{MainResult}, all facets of $\conv (S \cup Y_M)$ for which $y$ is an extreme point are nondominated. 
Overall, all facets of $\conv (S \cup Y_M)$ are nondominated, except the facet $F_M := \conv\{m^1,\ldots,m^p\}$
whose extreme points are all dummy points.
\end{proof}\medskip

We can now propose a first improvement of Algorithm~\ref{Balloon}. This algorithm initializes the set $S$ of known supported points with one nondominated supported point (for example using the weight $(\frac{1}{p},\ldots,\frac{1}{p})$). To complete the initialization, the set of dummy points $Y_M$ is computed. $\conv(S \cup Y_M)$ is therefore a full-dimensional polytope and Proposition~\ref{PropDumm} implies that only one of its facets is {currently} dominated, namely $F_M =  \conv\{m^1,\ldots,m^p\}$. Next, the main loop (lines 2-12) of Algorithm~\ref{Balloon} could simply be used to compute the set $\conv (Y \cup Y_M)$, but this would result again in the computation of dominated points. To avoid this, Algorithm~\ref{DumDicho} computes $\conv (Y_{SN1} \cup Y_M)$ (so that $S$ contains the set of nondominated extreme points at termination of the algorithm). This can be realized with a slight modification in the lines 2-12 of Algorithm~\ref{Balloon}: the facet $F_M$ is not considered in line 4 of {Algorithm~\ref{DumDicho}}.
Proposition~\ref{propDumDicho} justifies the correctness of the algorithm.

\begin{algorithm}[h]
\caption{Algorithm {\tt Dummy\_Dichotomy}}
\label{DumDicho}
\begin{algorithmic}[1]
\REQUIRE An instance of an MOILP
\ENSURE $S$ contains the set of nondominated extreme points of $\conv Y$
\STATE Initialize $S$ with one supported point and determine the set of dummy points $Y_M$
\REPEAT
\STATE Compute/Update $\conv(S \cup Y_M)$
\STATE Choose any unexplored facet $F$ to obtain a weight $\lambda$ defined by its normal (except the facet defined by $\{m^1,\ldots,m^p\}$)
\STATE Solve the weighted sum problem with objective function $\lambda^Tz$ 
\STATE Let $x$ be the optimal solution, $y := z(x)$ and $y'$ be a point of $F$
\IF{$\lambda^Ty' = \lambda^Ty$}
\STATE A (part of a) facet of $\conv (Y_{SN1} \cup Y_M)$ is identified
\ELSE  
\STATE A new nondominated supported point is found and added to $S$
\ENDIF
\UNTIL{All nondominated facets of $\conv(S \cup Y_M)$ have been explored}
\end{algorithmic}
\end{algorithm}

\begin{prop}\label{propDumDicho}
Algorithm~\ref{DumDicho} generates only supported nondominated points and $S$ contains at termination the set of nondominated extreme points of $\conv Y$.
\end{prop}
\begin{proof}
As only nondominated facets of $\conv (S \cup Y_M)$ are considered to define weights, only weights $\lambda \in \mR^p_>$ are exploited in the algorithm. These weights allow either to identify a (part of a) facet of $\conv Y$ (line 8) or to find a new supported nondominated point that is added to $S$ (line 10). After the initialization and after each update of $S$, the polytope $\conv (S \cup Y_M)$ is full-dimensional, and only one of its facets is dominated (the one defined by the dummy points) according to Proposition~\ref{PropDumm}. 

It remains to show that at termination of Algorithm~\ref{DumDicho}, $S$ contains all nondominated extreme points.  At the end of this algorithm, all nondominated facets of $\conv (S \cup Y_M)$ are identified as nondominated facets of $\conv (Y_{SN1} \cup Y_M)$. Consequently, for all $y \in S$, all extreme points of $W^{0M}_S(y)$ are identified as points of $W^{0M}(y)$. As for all $y \in S$, $W^{0M}(y) \subseteq W^{0M}_S(y)$, we have therefore $W^{0M}_S(y) = W^{0M}(y)$. The weight set decomposition is therefore complete and   
Proposition~\ref{compW0(y)}(5) implies that $S$ contains all nondominated extreme points.
\end{proof}\medskip

Algorithm~\ref{DumDicho} is only correct with an appropriate initialization, and in particular with an appropriate choice of dummy points. However, the use of dummy points has one main drawback: the large value of $M$ used as a component of the dummy points may be the cause of numerical imprecisions in the implemented algorithms. Hence, we propose next another improved version of Algorithm~\ref{Balloon}. In particular, another initialization is necessary. Figure~\ref{Decomp4} from Example~\ref{exdumm} illustrates the role of the dummy points. These points are the only extreme points $y$ of $\conv (Y \cup Y_M)$ such that $\cl(W^{0M}(y))$ has a  facet in $\fr(W^0)$. Another set of points with the same property is necessary to replace the dummy points. Hence, a natural idea is to determine all nondominated extreme points $y$ such that $\cl(W^0(y))$ has a facet in $\fr(W^0)$, i.e. the nondominated extreme points of the $p$ associated problems with the $(p-1)$ objectives given by $(z^1,\ldots,z^{k-1},z^{k+1},\ldots,z^p)$ {for $k=1,\dots,p$}. 
As one objective 
function is ignored, these points may correspond to several possible points for the $p$-objective problem, some of which may be dominated. According to Proposition~\ref{sup1}, these points are nondominated extreme points of the $p$-objective problem if and only if they are nondominated. 

\begin{prop}[\citealp{3NE}]\label{sup1}
Let $P$ be a problem with $p$ objectives and let $P_I$ be a subproblem
given by objectives indexed by $I \subset \{1,\ldots,p\}$. Let $y$ be {a point in $P$ such that its projection onto $P_I$ is} a nondominated extreme point of $P_I$. 
Then if $y$ is a nondominated point of $P$ it is also a nondominated extreme point of $P$.
\end{prop}

{Note that for every nondominated extreme point of a $(p-1)$-objective subproblem, at least one of the corresponding points in the $p$-objective problem is indeed nondominated.} 
We propose thus to initialize the set $S$ of supported points of our second algorithm with the nondominated extreme points that can be obtained using the $p$ associated subproblems with $(p - 1)$ objectives. {Each of these $(p-1)$-objective subproblems can again be solved recursively, i.e., by first initializing with all $(p-1)$ subproblems with $(p-2)$ objectives and then applying dichotomic search as in Algorithm~\ref{RecDicho}. Note, however, that this procedure involves the solution of $O(p!)$ subproblems, many of which are solved repeatedly.} In order to avoid redundancies, the solution of these subproblems should {thus} not be {implemented} recursively. {Indeed, it is sufficient to enumerate all subsets of objectives from $\{1,\dots,p\}$ which leaves us with a total of $O(2^p)$ subproblems to be solved.} After this initialization {phase}, there are obviously facets of $\conv S$ that are dominated. However, our purpose now is only to determine points of $Y_{SN1}$ that have not yet been found with the initialization. Proposition~\ref{facetRecDicho} states that these points always belong to (weakly) nondominated facets of $\conv S$.

\begin{prop}\label{facetRecDicho}
Let $S_{\init}$ be a set of supported points containing the nondominated extreme points that can be obtained using the $p$ associated subproblems with $(p - 1)$ objectives, and let $S$ be a set of supported nondominated points such that $S_{\init} \subset S$. For all $y \in S \setminus S_{\init}$, 
any facet $F$ of $\conv S$ for which $y$ is an extreme point 
is at least weakly nondominated.
\end{prop}
\begin{proof}
As the only points $s \in Y_{SN1}$ such that $\cl(W^0(s))$ has a facet in $\fr(W^0)$ are found in the initialization, all other points $y \in S\cap Y_{SN1}$ {satisfy} $W^0_{S_{\ext}}(y) \subset \cl(W^0)$. Indeed, suppose this is not the case, i.e. that for $y \in S \cap Y_{SN1}$, there is $\lambda \in W^0_{S_{\ext}}(y) \setminus \cl(W^0)$. We have therefore $\conv(\{\lambda\} \cup W^0_S(y)) \subseteq W^0_{S_{\ext}}(y)$. As $y \in Y_{SN1}$ then $\dim W^0_S(y) = p - 1$, and therefore $\conv(\{\lambda\} \cup W^0_S(y))$ intersects $\fr(W^0)$ in a polytope of dimension $p - 2$ (see Figure \ref{illustre} for an illustration with $p = 3$). This contradicts the assumption that $y \not \in S_{\init}$.
Proposition~\ref{MainResult} implies next that any facet for which $y$ is an extreme point is at least weakly nondominated.
\end{proof}\medskip

\begin{figure}
\center
\includegraphics[width = 0.5\textwidth]{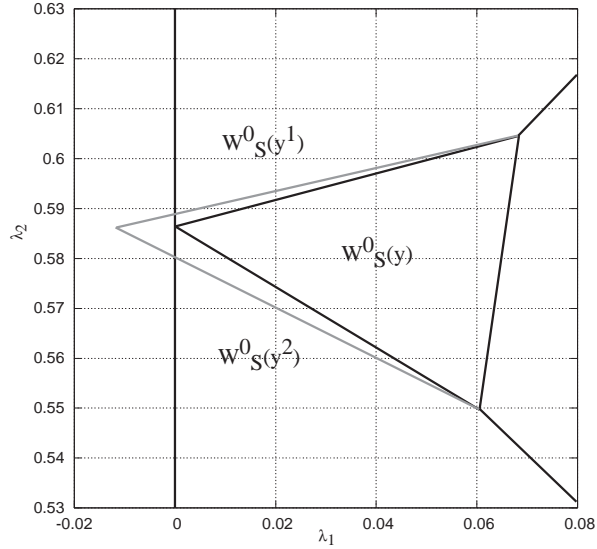}
\caption{If there is $\lambda \in W^0_{S_{\ext}}(y) \setminus \cl(W^0)$ then $\conv(\{\lambda\} \cup W^0_S(y))$ intersects $\fr(W^0)$ in a polytope of dimension $p - 2$.}
\label{illustre}
\end{figure}

Algorithm~\ref{RecDicho} is another variation of Algorithm~\ref{Balloon}, using the initialization suggested by Proposition~\ref{facetRecDicho} and defining weights using only nondominated facets of $\conv S$. Proposition~\ref{propRecDicho} justifies the correctness of the algorithm.

\begin{algorithm}[h]
\caption{Algorithm {\tt Bd\_Dichotomy}}
\label{RecDicho}
\begin{algorithmic}[1]
\REQUIRE An instance of an MOILP
\ENSURE $S$ contains the set of nondominated extreme points of $\conv Y$
\STATE Initialize $S$ using the nondominated extreme points obtained from the $p$ associated ($p-1$)-objective subproblems
\REPEAT
\STATE Compute/Update $\conv S$
\STATE Choose any unexplored facet with a weight $\lambda \in \mR^p_>$
\STATE Solve the weighted sum problem with objective function $\lambda^Tz$ 
\STATE Let $x$ be the optimal solution, $y := z(x)$ and $y'$ be a point of $F$
\IF{$\lambda^Ty' = \lambda^Ty$}
\STATE A (part of a) nondominated facet of $\conv Y$ is identified
\ELSE  
\STATE A new nondominated supported point is found and added to $S$
\ENDIF
\UNTIL{All nondominated facets of $\conv S$ have been explored}
\end{algorithmic}
\end{algorithm}

\begin{prop}\label{propRecDicho}
Algorithm~\ref{RecDicho} generates only supported nondominated points
and $S$ contains at termination the set of nondominated extreme points of $\conv Y$.
\end{prop}

\begin{proof}
As only nondominated facets of $\conv S$ are considered to define weights (line 4), only weights $\lambda \in \mR^p_>$ are exploited in the algorithm. Therefore, only supported nondominated points are generated.

It remains to show that at termination of Algorithm~\ref{RecDicho}, $S$ contains all nondominated extreme points.  At the end of this algorithm, all nondominated facets of $\conv S$ are identified as nondominated facets of $\conv Y$. Consequently, for all $y \in S\setminus S_{\init}$, all extreme points of $W^0_S(y)$ that are located in $W^0$ are identified as points of $W^0(y)$. Next, for any $y \in S\setminus S_{\init}$, the extreme points of $W^0(y)$ that are located in $\fr(W^0)$ are necessarily common extreme points with $W^0(y')$ where $y' \in S_{\init}$, and these extreme points are identified with the computation of $S_{\init}$. Thus, it is not necessary to consider weakly nondominated facets in the main loop of Algorithm~\ref{RecDicho}. Finally, for all $y \in S$, $W^0(y) \subseteq W^0_S(y)$, we have therefore $W^0_S(y) = W^0(y)$. The weight set decomposition is therefore complete and  Proposition~\ref{compW0(y)}(5) allows to conclude that $S$ contains all nondominated extreme points.
\end{proof}\medskip

In the bi-objective case, the classical dichotomic scheme can be initialized using dominated points minimizing each objective. Next, the lexicographically optimal points dominating the initial dominated points will be found. The same way, it is possible to initialize the set of points $S$ of Algorithm~\ref{RecDicho} using nondominated extreme points of the $p$ associated problems with $(p-1)$ objectives, that may be dominated for the initial $p$-objective problem. Using such a set $S_{\init}$ of initial points, we keep the property that the points $s$ of $S_{\init}$ are the only points such that $W^0_S(s)$ has a facet in $\fr(W^0)$. Algorithm~\ref{RecDicho} can therefore be applied and the nondominated extreme points dominating initial points will necessarily be found (as no dominated point can be identified as an extreme point of a nondominated facet of $\conv Y_{SN}$).

\subsection{Complexity of the algorithms}

Using the complexity results stated by \citet{boekler15}, we show that our algorithm runs also in incremental polynomial time under the same assumptions:
\begin{itemize}
\item The number of dimensions $p$ is fixed,
\item The lexicographic single-objective problem can be solved in polynomial time,
\item The convex hull is not computed with an incremental convex hull algorithm, but with a statical convex hull algorithm \citep[e.g.][]{Chazelle93}.
\end{itemize}
The same way as \citet{boekler15}, an implementation using an incremental convex hull algorithm cannot be guaranteed to run in incremental polynomial time (as it has been shown that incremental convex hull algorithms are not output sensitive \citep{Bremner99}). However, the use of an incremental convex hull algorithm remains an appropriate choice to obtain a practically efficient implementation.

\medskip

One iteration of the algorithm proposed by \citep{boekler15} is composed of the choice of an extreme point of the dual polyhedron $\cal D$ to obtain a (yet) unexplored weight $\lambda$, the solution of the obtained weighted sum problem, and the possible update of the dual polyhedron $\cal D$. Depending on the order of choice of the weights, the available weights for the next iterations may vary, as well as the number of single-objective problems to solve to complete the execution of the method. However, the complexity statement does not depend on the order of choice of the weights.

\medskip

The initial polyhedron $\cal D$ gives the possibility to choose among the extreme points of $\cl(W^0)$. Suppose that we start by choosing any possible weight such that $\lambda_p = 0$ (initially any weight except $(0,\ldots,0,1)$), and that we repeat this choice until there is no {choice left}. Then as the update of the dual polyhedron $\cal D$ describes a weight set decomposition with the $(p-1)$ first components of its extreme points, the nondominated extreme points of the $(p-1)$-objective problem defined by the objectives $z_1,\ldots,z_{p-1}$ are therefore computed. Next, we can in the same way choose the weights to determine the nondominated extreme points of the other $(p-1)$-objective problems. Hence, this ordering in the initialization phase of  the algorithm corresponds exactly to the initialization of Algorithm~\ref{RecDicho}. Besides this initialization, Algorithm~\ref{RecDicho} only explores weights defined using normals to nondominated facets, that correspond exactly to extreme points of $W^0_S(y)$ in $W^0$ for all $y \in S$, according to Proposition~\ref{faces}. In other words, these weights correspond exactly to extreme points of the dual polyhedron $\cal D$ that define yet unexplored weights for the algorithm by \citet{boekler15}, and both algorithms may continue with the solution of the same weighted sum problems. Consequently, Algorithm~\ref{RecDicho} and the algorithm proposed by \citet{boekler15} can be seen as dual to each other, and share the same theoretical complexity.

\medskip

We can obtain the same conclusion for Algorithm~\ref{DumDicho}. Indeed, Algorithm~\ref{DumDicho} can be seen as an application of Algorithm~ \ref{RecDicho} with an initialization based on the introduction of $p$ dummy points in the set of nondominated extreme points. As the number of dimensions $p$ is considered as fixed in the theoretical complexity, the same conclusion is valid for Algorithm~\ref{DumDicho}.

%% file: experiments.tex
\section{Experimental Results}\label{sec:exp}

For our experiments, we have used instances of 3AP and 3KP from \citep{3SE}. The instances of 3AP have been extended up to a size of $100 \times 100$. We have also generated instances of 4AP and 4KP.

In the following, we use 10 series of 10 instances of 3AP with a size varying
from $10 \times 10$ to $100 \times 100$ with a step of 10, 8 series of 10 instances of 3KP with a size varying from 50 to 400 with a step of 50, 8 series of 10 instances of 4AP with a size varying from $5 \times 5$ to $40 \times 40$ with a step of 5, 4 series of 10 instances of 4KP with a size varying from $50 $ to $200$ items with a step of 50.
The objective function coefficients  of 3AP and 4AP instances are generated randomly in $[0,20]$, and the objective function and constraint coefficients of 3KP and 4KP instances are generated in $[1,100]$. The hungarian method and a dynamic programming method have been used to solve instances of respectively AP and KP \citep{PS82}. The computational geometry library CGAL \citep{cgal} has been used for the incremental computation of convex hull.

The computer used for the experiments is equipped with a Intel Core i7-4930MX Extreme 3 Ghz processor with 16 Gb of RAM, and runs under Ubuntu 14.04 LTS.
The proposed algorithms have been implemented in C++, and the binary have been obtained using the compiler g++ with the optimizer option -O2. 
We denote next the used implementations:
\begin{itemize}
\item the original implementation of the method by \citet{3SE} (3AP and 3KP instances only, denoted next by $PGE$),  
\item two implementations of Algorithm \ref{DumDicho} ($v1_{ex}$: convex hull computed using exact arithmetic, $v1_{fl}$: convex hull computed using floating numbers), 
\item two implementations of Algorithm \ref{RecDicho} ($v2_{ex}$: convex hull computed using exact arithmetic, $v2_{fl}$: convex hull computed using floating numbers),
\end{itemize}

In the case of the implementations of Algorithms \ref{DumDicho} and \ref{RecDicho} using exact arithmetic for the computation of the convex hull, the weight given by the normal to a facet is first divided by the greatest common divisor of its components. Next, the single-objective problem is solved using integer numbers if possible, otherwise using floating numbers (as very large numbers may be obtained). Here, the goal is not to obtain the fastest implementation, but the most reliable possible. If floating numbers are used for the computation of the convex hull, then all single-objective problems are necessarily solved using floating numbers. 

The initialization of Algorithm \ref{RecDicho} for 3AP and 3KP instances is realized by computing nondominated extreme points of the three bi-objective problems with objective functions $(z_1,z_2)$, $(z_2,z_3)$ and $(z_1,z_3)$ respectively, that are nondominated extreme points of the three-objective problem. For this, weighted-sum problems of the kind $\lambda_1z_1 + \lambda_2z_2$ are replaced by $M\lambda_1z_1 + M\lambda_2z_2 + z_3$ where an appropriate value of $M$ can be fixed using the fact that objective coefficients are integer and using the range of data (see, for example, \citep{Ozlen}). The same is done for the other bi-objective problems. Finally, the single-objective problems considered here are also solved using integer numbers if possible, otherwise using floating numbers. For 4AP and 4KP instances, nondominated extreme points of the four problems with objective functions $(z_1,z_2,z_3)$, $(z_1,z_2,z_4)$, $(z_1,z_3,z_4)$ and $(z_2,z_3,z_4)$ are in practice computed without guarantee to be nondominated points 
of the four-objective problem, as the scales of the weights applied to objective functions may vary considerably. 

Finally, the implementations of the method by \citet{3SE} use only floating numbers to define weights. 

\medskip

After the application of any of these implementations, the convex hull of the obtained points is reconstructed using the qhull library \citep{Barber:1996:QAC:235815.235821}, in order to keep only nondominated extreme points. Potential non-extremal supported points and nonsupported or dominated points (that may be obtained due to numerical imprecisions) are therefore filtered. Given the implementation choices, the most reliable implementation is clearly the one denoted $v2_{ex}$, as it uses integer numbers whenever possible and does not rely on dummy points. We will therefore consider the results of this implementation as a reference in the following. 

\medskip

All results are summarized in Tables \ref{tab3AP}-\ref{tab4KP}. In the columns showing the number of calls to a single-objective solver, corresponding to the columns $v1_{ex}$, $v2_{ex}$, $v2_{fl}$, the number between brackets gives the percentage of calls to a solver using floating numbers. This precision is useless for the columns $v1_{fl}$ and $PGE$ since it would always be 100\%. We can notice that the $v2_{ex}$ implementations do not require any call to a single-objective solver with floating numbers for any 3AP and 4AP instance. This is not the case for 3KP and 4KP instances, in particular because of a larger range of data. Therefore, we cannot claim that the $v2_{ex}$ implementation is perfectly reliable for these instances. However, the percentage of calls to a solver with floating numbers is not a perfect indicator of reliability. Indeed, a call to a solver using floating numbers indicates only that the components of the weight are too large for the solution to be completed only with integer 
numbers. If these components are just slightly too large, the solution will most likely give a correct result. If some of these components are in a larger scale of magnitude, numerical imprecisions will most likely occur. For example in the column $v1_{ex}$ of Table \ref{tab3AP}, the percentage of single-objective problems solved using floating numbers decreases gradually with the size. However, as all solutions are computed using integer numbers in the $v2_{ex}$ case, we can deduce that all solutions computed using floating numbers in the $v1_{ex}$ case concern facets, one extreme point of which is a dummy point. Finally, the used value of $M$ is larger and larger with the size of the problem. Therefore, the $v1_{ex}$ implementation should become less and less reliable with the size of the problem.

\medskip

Some information does not appear in Tables \ref{tab3AP}-\ref{tab4KP}. Indeed, the column $|Y_{SN_1}|$ is only provided by the implementation $v2_{ex}$ and this number may slightly differ for other implementations. For 3AP and 3KP instances, we always obtain the same set $Y_{SN_1}$. 
However, for 4AP instances, the $v1_{ex}$ and $v1_{fl}$ implementations sometimes miss (between 1 and 3) nondominated extreme points. It concerns one instance of size $15 \times 15$, one instance of size $20 \times 20$, two instances of size $25 \times 25$, one instance of size $30 \times 30$, seven instances of size $35 \times 35$, two instances of size $40 \times 40$. One nondominated extreme point is also missed by the implementation $v2_{fl}$ for one instance of size $40 \times 40$. For 4KP instances, all implementations always obtain the same set $Y_{SN_1}$ for all 
instances. However, the $v1_{fl}$ implementation is only able to solve 2 instances of size 150, and 2 instances of size 200. For other instances of the same size, no results are obtained due to numerical instabilities. 

\medskip


Finally, the proposed method is faster than $PGE$ whatever the used implementation and is able to solve instances with more than three objectives. The $v1$ and $v2$ implementations have similar computational times for 3AP and 3KP, and the floating versions are slightly faster for 3KP. For 4AP and 4KP instances, the computational times completely differ. The computation of the convex hull becomes a lot more expensive. Moreover, the $v_2$ versions are here slower than the $v_1$ versions, this is due to a (slightly) larger number of facets in the computed convex hulls. Finally, there is a large difference in computational time for the computation of convex hull using exact arithmetic and floating numbers. 

\begin{landscape}
\begin{table} 
\center  
\begin{tabular}{r || r | r r r r r | r r r r r} 
\hline
& &  \multicolumn{5}{c|}{Number of calls to a single-objective solver} &
  \multicolumn{5}{c}{CPU time (s)}\\ 
\hline 
size & $|Y_{SN_1}|$ & $PGE$ & $v1_{ex}$ & $v2_{ex}$ & $v1_{fl}$ & $v2_{fl}$ & $PGE$ & $v1_{ex}$ & $v2_{ex}$ & $v1_{fl}$ & $v2_{fl}$\\
\hline
$10 \times 10$ & 37.2 & 245.2 & 112.6 (0,6\%) & 111.1 (0\%) & 112.6 & 111.1 (62.4\%) & 0.007 & 0.004 & 0.001 & 0.000 & 0.000\\
$20 \times 20$ & 156.4 & 1079 & 470.2 (11.7\%) & 464.1 (0\%) & 470.2 & 464.1 (81.4\%) & 0.028 & 0.037 & 0.039 & 0.018 & 0.017\\
$30 \times 30$ & 351.9 & 2448.8 & 1057 (11.7\%) & 1044.6 (0\%) & 1057.3 & 1044.9 (87\%) & 0.157 & 0.137 & 0.148 & 0.094 & 0.092\\
$40 \times 40$ & 646.6 & 4571 & 1941.1 (9.5\%) & 1921.9 (0\%) & 1940.8 & 1921.9 (90.9\%) & 0.614 & 0.405 & 0.438 & 0.344 & 0.338\\
$50 \times 50$ & 1061.5 & 7500.5 & 3187.9 (7.8\%) & 3159.8 (0\%) & 3187.6 & 3159.6 (93\%) & 1.811 & 1.051 & 1.120 & 1.000 & 0.986\\
$60 \times 60$ & 1514.3 & $\times$ & 4544.5 (6.5\%) & 4507.9 (0\%) & 4546.3 & 4508.2 (94.5\%) & $\times$ & 2.187 & 2.291 & 2.222 & 2.201\\
$70 \times 70$ & 2023.3 & $\times$ & 6076 (6\%) & 6020.7 (0\%) & 6075.1 & 6021.9 (95.3\%) & $\times$ & 4.080 & 4.286 & 4.644 & 4.464\\
$80 \times 80$ & 2580.4 & $\times$ & 7752.7 (5.3\%) & 7678.4 (0\%) & 7750.6 & 7680.2 (96.2\%) & $\times$ & 7.083 & 7.280 & 8.144 & 8.031\\
$90 \times 90$ & 3260.8 & $\times$ & 9794.5 (4.7\%) & 9704.2 (0\%) & 9791.5 & 9703.6 (96.8\%) & $\times$ & 11.598 & 11.767 & 14.174 & 13.632\\
$100 \times 100$ & 3824.2 & $\times$ & 11487.4 (4.3\%) & 11384.8 (0\%) & 11485.9 & 11381.5 (97.2\%) & $\times$ & 17.158 & 17.275 & 21.528 & 20.822\\
\hline
\end{tabular}
\caption{Number of nondominated extreme points, number of calls to the single-objective solver and CPU time for 3AP instances. The numbers between brackets indicates the percentage of calls to a single-objective solver using floating numbers.} 
\label{tab3AP}
\end{table}
\end{landscape}

\begin{landscape}
\begin{table} 
\center  
\begin{tabular}{r || r | r r r r r | r r r r r} 
\hline
& &  \multicolumn{5}{c|}{Number of calls to a single-objective solver} &
  \multicolumn{5}{c}{CPU time (s)}\\ 
\hline 
size & $|Y_{SN_1}|$ & $PGE$ & $v1_{ex}$ & $v2_{ex}$ & $v1_{fl}$ & $v2_{fl}$ & $PGE$ & $v1_{ex}$ & $v2_{ex}$ & $v1_{fl}$ & $v2_{fl}$\\
\hline
50 & 58.1 & 387.9 & 175.3 (32.9\%) & 175.6 (35.2\%) & 175.3  & 175.6 & 0.009 & 0.014 & 0.012 & 0.004 & 0.004\\
100 & 173.4 & 1242.9 & 521.3 (20\%) & 520.7 (20.8\%) & 521.3 & 520.7 & 0.057 & 0.062 & 0.063 & 0.033 & 0.032\\
150 & 332 & 2448.9 & 997 (15.2\%) & 995.8 (15.7\%) & 997 & 995.8 & 0.186 & 0.159 & 0.164 & 0.101 & 0.102\\
200 & 629 & 4720 & 1888 (11.3\%) & 1885.6 (11.5\%) & 1888 & 1885.6 & 0.504 & 0.42 & 0.432 & 0.294 & 0.296\\
250 & 822.1 & 6195.2 & 2467.3 (10.6\%) & 2465.9 (11\%) & 2467.3 & 2465.9 & 0.779 & 0.634 & 0.669 & 0.466 & 0.468\\
300 & 1177.6 & 8969.9 & 3533.8 (8.9\%) & 3531.2 (9.1\%) & 3533.8 & 3531.2 & 1.464 & 1.156 & 1.230 & 0.896 & 0.912\\
350 & 1695.2 & 12988.5 & 5086.6 (7.8\%) & 5083.4 (7.8\%) & 5086.6 & 5083.4 & 2.574 & 2.104 & 2.234 & 1.673 & 1.704\\
400 & 2070.2 & 15897.3 & 6211.6 (7.7\%) & 6204.9 (7.7\%) & 6211.6 & 6204.9 & 3.506 & 2.956 & 3.082 & 2.385 & 2.412\\
\hline
\end{tabular}
\caption{Number of nondominated extreme points, number of calls to the single-objective solver and CPU time for 3KP instances.  The numbers between brackets indicates the percentage of calls to a single-objective solver using floating numbers.} 
\label{tab3KP}
\end{table}
\end{landscape}

\begin{landscape}
\begin{table} 
\center  
\small
\begin{tabular}{r || r | r r r r | r r r r} 
\hline
& &  \multicolumn{4}{c|}{Number of calls to a single-objective solver} &
  \multicolumn{4}{c}{CPU time (s)}\\ 
\hline 
size & $|Y_{SN_1}|$ & $v1_{ex}$ & $v2_{ex}$ & $v1_{fl}$ & $v2_{fl}$ & $v1_{ex}$ & $v2_{ex}$ & $v1_{fl}$ & $v2_{fl}$\\
\hline
$5 \times 5$ & 15.5 & 81.1 (5.9\%) & 85 (0\%) [65] & 81.1 & 85 (58.9\%) [65]  & 0.012 & 0.012 [0.002] & 0 & 0 [0]\\
$10 \times 10$ & 128.4 & 761.5 (29.1\%) & 766 (0\%) [352.6] & 761.5 & 766 (89.3\%) [352.6]  & 0.141 & 0.269 [0.018] & 0.042 & 0.055 [0.002]\\
$15 \times 15$ & 422 & 2622 (28.7\%) & 2621.5 (0\%) [893.3] & 2622 & 2621.3 (95.1\%) [893.3] & 0.864 & 2.186 [0.068] & 0.412 & 0.589 [0.018]\\
$20 \times 20$ & 970.8 & 6205.5 (23.8\%) & 6198.4 (0\%) [1602.8] & 6205.8 & 6198.6 (97.2\%) [1602.8] & 3.748 & 9.593 [0.146] & 2.296 & 3.140 [0.063]\\
$25 \times 25$ & 2034.7 & 13282 (19.2\%) & 13263 (0\%) [2674.2] & 13283.1 & 13264.1 (98.3\%) [2674.2] & 14.970 & 117.647 [0.596] & 10.735 & 14.374 [0.171]\\
$30 \times 30$ & 3393.4 & 22337.5 (17\%) & 22296 (0\%) [3953] & 22336 & 22295.3 (98.8\%) [3953] & 38.189 & 117.647 [0.596] & 29.177 & 42.312 [0.364]\\
$35 \times 35$ & 5962.5 & 39693.7 (14.1\%) & 39611.4 (0\%) [5717.5] & 39692 & 39610.6 (99.2\%) [5717.8] & 119.790 & 356.330 [1.120] & 97.410 & 146.630 [0.762]\\
$40 \times 40$ & 8612.8 & 57610.0 (12.7\%) & 57482.5 (0\%) [7444.7] & 57611.4 & 57479.8 (99.4\%) [7445.1] & 280.009 & 770.816 [1.796] & 235.548 & 357.133 [1.374]\\
\hline
\end{tabular}
\caption{Number of nondominated extreme points, number of calls to the single-objective solver and CPU time for 4AP instances.  The numbers between brackets indicates the percentage of calls to a single-objective solver using floating numbers. The numbers between hooks indicates respectively the number of calls to a single-objective solver and the CPU time needed for the initialization of the methods $v2_{ex}$ and $v2_{fl}$.} 
\label{tab4AP}
\end{table}
\end{landscape}

\begin{landscape}
\begin{table} 
\center  
\small
\begin{tabular}{r || r | r r r r | r r r r} 
\hline
& &  \multicolumn{4}{c|}{Number of calls to a single-objective solver} &
  \multicolumn{4}{c}{CPU time (s)}\\ 
\hline 
size & $|Y_{SN_1}|$ & $v1_{ex}$ & $v2_{ex}$ & $v1_{fl}$ & $v2_{fl}$ & $v1_{ex}$ & $v2_{ex}$ & $v1_{fl}$ & $v2_{fl}$\\
\hline
50 & 160.7 & 927.4 (69.3\%) & 943.9 (30.9\%) [487.4] & 927.4  & 943.9 [487.4]  & 0.232 & 0.436 [0.036] & 0.076 & 0.098 [0.012]\\
100 & 932.4 & 5778.6 (76.3\%) & 5802.5 (49.4\%) [1806.5] & 5778.6 & 5802.5 [1806.5] & 4,008 & 9.764 [0.203] & 2.237 & 2.968 [0.113]\\
150 & 2818 & 17919.8 (83.5\%) & 17956.5 (63.4\%) [3965.2] & $\times$ & 17965.5 [3965.2] & 29.262 & 83.907 [0.657] & $\times$ & 27.4 [0.405]\\
200 & 5991 & 38728.3 (86.9\%) & 38772.7 (70.8\%) [6704] & $\times$ & 38772.7 [6704] & 135.613 &  373.004 [1.496] & $\times$ & 144.565 [1.004]\\
\hline
\end{tabular}
\caption{Number of nondominated extreme points, number of calls to the single-objective solver and CPU time for 4KP instances.  The numbers between brackets indicates the percentage of calls to a single-objective solver using floating numbers. The numbers between hooks indicates respectively the number of calls to a single-objective solver and the CPU time needed for the initialization of the methods $v2_{ex}$ and $v2_{fl}$.} 
\label{tab4KP}
\end{table}
\end{landscape}

%% file: ConvexSupported.bbl
\begin{thebibliography}{33}
\providecommand{\natexlab}[1]{#1}
\providecommand{\url}[1]{\texttt{#1}}
\expandafter\ifx\csname urlstyle\endcsname\relax
  \providecommand{\doi}[1]{doi: #1}\else
  \providecommand{\doi}{doi: \begingroup \urlstyle{rm}\Url}\fi

\bibitem[cga()]{cgal}
\textsc{Cgal}, {C}omputational {G}eometry {A}lgorithms {L}ibrary.
\newblock http://www.cgal.org.

\bibitem[Aneja and Nair(1979)]{aneja79}
Y.P. Aneja and K.P.K. Nair.
\newblock Bicriteria transportation problem.
\newblock \emph{Management Sci.}, 25:\penalty0 73--78, 1979.

\bibitem[Barber et~al.(1996)Barber, Dobkin, and
  Huhdanpaa]{Barber:1996:QAC:235815.235821}
C.B. Barber, D.P. Dobkin, and H.~Huhdanpaa.
\newblock The quickhull algorithm for convex hulls.
\newblock \emph{ACM Trans. Math. Softw.}, 22\penalty0 (4):\penalty0 469--483,
  December 1996.
\newblock ISSN 0098-3500.
\newblock \doi{10.1145/235815.235821}.
\newblock URL \url{http://doi.acm.org/10.1145/235815.235821}.

\bibitem[Benson(1998{\natexlab{a}})]{benson97}
H.P. Benson.
\newblock An outer approximation algorithm for generating all efficient extreme
  points in the outcome set of a multiple objective linear programming problem.
\newblock \emph{J. of Global Optim.}, 13\penalty0 (1):\penalty0 1--24,
  1998{\natexlab{a}}.

\bibitem[Benson(1998{\natexlab{b}})]{benson98a}
H.P. Benson.
\newblock Further analysis of an outcome set-based algorithm for
  multiple-objective linear programming.
\newblock \emph{J. of Optim. Theory and Appl.}, 97\penalty0 (1):\penalty0
  1--10, 1998{\natexlab{b}}.

\bibitem[Benson and Sun(2000)]{benson2000}
H.P. Benson and E.~Sun.
\newblock Outcome space partition of the weight set in multiobjective linear
  programming.
\newblock \emph{J. of Optim. Theory and Appl.}, 105\penalty0 (1):\penalty0
  17--36, 2000.

\bibitem[Benson and Sun(2002)]{benson2002}
H.P. Benson and E.~Sun.
\newblock A weight set decomposition algorithm for finding all efficient
  extreme points in the outcome set of a multiple objective linear program.
\newblock \emph{Eur. J. of Oper. Res.}, 139:\penalty0 26--41, 2002.

\bibitem[B{\"o}kler and Mutzel(2015)]{boekler15}
F.~B{\"o}kler and P.~Mutzel.
\newblock Output-sensitive algorithms for enumerating the extreme nondominated
  points of multiobjective combinatorial optimization problems.
\newblock In N.~Bansal and I.~Finocchi, editors, \emph{ESA 2015}, volume 9294
  of \emph{Lecture Notes in Computer Science}, pages 288--299. Springer Verlag,
  Berlin Heidelberg, 2015.

\bibitem[Bremner(1999)]{Bremner99}
D.~Bremner.
\newblock Incremental convex hull algorithms are not output sensitive.
\newblock \emph{Discrete \& Comput. Geometry}, 21:\penalty0 57--68, 1999.

\bibitem[Chazelle(1993)]{Chazelle93}
B.~Chazelle.
\newblock An optimal convex hull algorithm in any fixed dimension.
\newblock \emph{Discrete \& Comput. Geometry}, 10:\penalty0 377--409, 1993.

\bibitem[Cohon(1978)]{cohon78}
J.L. Cohon.
\newblock \emph{Multiobjective Programming and Planning}.
\newblock Academic Press, New York, 1978.

\bibitem[Csirmaz(2018)]{csirmaz2018inner}
Laszlo Csirmaz.
\newblock Inner approximation algorithm for solving linear multiobjective
  optimization problems, 2018.

\bibitem[Ehrgott and Gandibleux(2002)]{MEXG2002}
M.~Ehrgott and X.~Gandibleux.
\newblock Multiobjective combinatorial optimization.
\newblock In M.~Ehrgott and X.~Gandibleux, editors, \emph{Multiple Criteria
  Optimization: State of the Art Annotated Bibliographic Surveys}, volume~52 of
  \emph{Kluwer's International Series in Operations Research and Management
  Science}, pages 369--444. Kluwer Academic Publishers, Boston, 2002.

\bibitem[Ehrgott and Wiecek(2005)]{EhrgottWiecekChapter}
M.~Ehrgott and M.~Wiecek.
\newblock Multiobjective programming.
\newblock In J.~Figueira, S.~Greco, and M.~Ehrgott, editors,
  \emph{Multicriteria Decision Analysis: State of the Art Surveys}, pages
  667--722. Springer Science + Business Media, New York, 2005.

\bibitem[Ehrgott et~al.(2007)Ehrgott, Puerto, and Rodriguez-Chia]{Ehrgott:2007}
M.~Ehrgott, J.~Puerto, and A.~Rodriguez-Chia.
\newblock Primal-dual simplex method for multiobjective linear programming.
\newblock \emph{J. of Optim. Theory and Appl.}, 134\penalty0 (3):\penalty0
  483--497, 2007.

\bibitem[Ehrgott et~al.(2012)Ehrgott, L{\"o}hne, and Shao]{ehrg:adua:2012}
M.~Ehrgott, A.~L{\"o}hne, and L.~Shao.
\newblock A dual variant of {B}ensons's outer approximation algorithm for
  multiple objective linear programming.
\newblock \emph{J. Global Optim.}, 52:\penalty0 757--778, 2012.

\bibitem[Ehrgott et~al.(2016)Ehrgott, Gandibleux, and Przybylski]{Ehrgott2016}
M.~Ehrgott, X.~Gandibleux, and A.~Przybylski.
\newblock \emph{Exact Methods for Multi-Objective Combinatorial Optimisation},
  pages 817--850.
\newblock Springer New York, New York, NY, 2016.
\newblock ISBN 978-1-4939-3094-4.

\bibitem[Hamel et~al.(2014)Hamel, L{\"o}hne, and Rudloff]{hame:bens:2014}
A.H. Hamel, A.~L{\"o}hne, and B.~Rudloff.
\newblock Benson type algorithms for linear vector optimization and
  applications.
\newblock \emph{J. Global Optim.}, 59:\penalty0 811--836, 2014.

\bibitem[Heyde and L{\"o}hne(2008)]{heyd:geom:2008}
F.~Heyde and A.~L{\"o}hne.
\newblock Geometric duality in multiple objective linear programming.
\newblock \emph{SIAM J. Optim.}, 19:\penalty0 836--845, 2008.

\bibitem[Isermann(1974)]{iser74}
H.~Isermann.
\newblock Proper efficiency and linear vector maximum problem.
\newblock \emph{Operations Research}, 22:\penalty0 189--191, 1974.

\bibitem[\"{O}zlen and Azizo\u{g}lu(2009)]{Ozlen}
M.~\"{O}zlen and M.~Azizo\u{g}lu.
\newblock Multi-objective integer programming: A general approach for
  generating all non-dominated solutions.
\newblock \emph{European Journal of Operational Research}, 199\penalty0
  (1):\penalty0 25 -- 35, 2009.

\bibitem[\"Ozpeynirci(2008)]{OZ08}
\"O. \"Ozpeynirci.
\newblock \emph{Approaches for multiobjective combinatorial optimization
  problems}.
\newblock PhD thesis, Middle East Technical University, Turkey, 2008.

\bibitem[\"Ozpeynirci and K\"oksalan(2010)]{MSOzpKok}
\"O. \"Ozpeynirci and M.~K\"oksalan.
\newblock An exact algorithm for finding extreme supported nondominated points
  of multiobjective mixed integer programs.
\newblock \emph{Management Sci.}, 56:\penalty0 2302--2315, 2010.

\bibitem[Papadimitriou and Steiglitz(1982)]{PS82}
C.H. Papadimitriou and K.~Steiglitz.
\newblock \emph{Combinatorial Optimization}.
\newblock Prentice Hall, Englewood Cliffs, 1982.

\bibitem[Przybylski(2006)]{TheseAnthony}
A.~Przybylski.
\newblock \emph{M\'ethode en deux phases pour la r\'esolution exacte de
  probl\`emes d'optimisation combinatoire comportant plusieurs objectifs :
  nouveaux d\'eveloppements et application au probl\`eme d'affectation
  lin\'eaire}.
\newblock PhD thesis, Universit\'e de Nantes, 2006.

\bibitem[Przybylski et~al.(2010{\natexlab{a}})Przybylski, Gandibleux, and
  Ehrgott]{3NE}
A.~Przybylski, X.~Gandibleux, and M.~Ehrgott.
\newblock A two phase method for multi-objective integer programming and its
  application to the assignment problem with three objectives.
\newblock \emph{Discrete Optim.}, 7:\penalty0 149--165, 2010{\natexlab{a}}.

\bibitem[Przybylski et~al.(2010{\natexlab{b}})Przybylski, Gandibleux, and
  Ehrgott]{3SE}
A.~Przybylski, X.~Gandibleux, and M.~Ehrgott.
\newblock A recursive algorithms for finding all nondominated extreme points in
  the outcome set of a multiobjective integer program.
\newblock \emph{INFORMS J. Comput.}, 22:\penalty0 371--386, 2010{\natexlab{b}}.

\bibitem[Rennen et~al.(2011)Rennen, van Dam, and den Hertog]{RenDamHer11}
G.~Rennen, E.R. van Dam, and D.~den Hertog.
\newblock Enhancement of sandwich algorithms for approximating
  higher-dimensional convex pareto sets.
\newblock \emph{INFORMS J. Comput.}, 23\penalty0 (4):\penalty0 493--517, 2011.
\newblock \doi{10.1287/ijoc.1100.0419}.

\bibitem[Ruzika and Wiecek(2005)]{RW05}
S.~Ruzika and M.~Wiecek.
\newblock Approximation methods in multiobjective programming.
\newblock \emph{J. of Optim. Theory and Appl.}, 126:\penalty0 473--501, 2005.

\bibitem[Schandl et~al.(2002)Schandl, Klamroth, and Wiecek]{SKW02}
B.~Schandl, K.~Klamroth, and M.~Wiecek.
\newblock Norm-based approximation in multicriteria programming.
\newblock \emph{Comput. and Math. with Appl.}, 44:\penalty0 925--942, 2002.

\bibitem[Steuer(1985)]{Steuer}
R.E. Steuer.
\newblock \emph{Multiple Criteria Optimization: Theory, Computation and
  Application}.
\newblock John Wiley \& Sons, New York, NY, 1985.

\bibitem[Tenfelde-Podehl(2003)]{2phrec}
D.~Tenfelde-Podehl.
\newblock A recursive algorithm for multiobjective combinatorial optimization
  problems with $q$ criteria.
\newblock Technical report, Institut f\"ur Mathematik, Technische Universit\"at
  Graz, 2003.

\bibitem[Yu and Zeleny(1975)]{yuzeleny}
P.L. Yu and M.~Zeleny.
\newblock The set of all nondominated solutions in linear cases and a
  multicriteria simplex method.
\newblock \emph{J. of Math. Anal. and Appl.}, 49:\penalty0 430--468, 1975.

\end{thebibliography}
